\documentclass[reqno]{amsart}

\usepackage{amsfonts}
\usepackage{amssymb}

\usepackage{amscd}
\usepackage{pictexwd,dcpic}
\usepackage{graphicx}
\usepackage{tikz}
\usepackage{fullpage}
\usepackage{hyperref}
\hypersetup{
    colorlinks=true,
    linkcolor=blue,
    filecolor=magenta,
    urlcolor=cyan,
    citecolor=cyan,}
\usepackage[nameinlink,capitalize]{cleveref}
\usepackage{enumitem}

\title{Tensor product surfaces and graded syzygies}

\author{Matthew Weaver}
\address{School of Mathematical and Statistical Sciences, Arizona State University, Wexler Hall, Tempe AZ 85281}
\email{matthew.j.weaver@asu.edu}

\date{}

\newtheorem{thm}{Theorem}[section]
\newtheorem*{thm-nonum}{Theorem}

\newtheorem{prop}[thm]{Proposition}
\newtheorem{lemma}[thm]{Lemma}
\newtheorem{cor}[thm]{Corollary}
\numberwithin{equation}{section}

\theoremstyle{definition}
\newtheorem{rem}[thm]{Remark}
\newtheorem{set}[thm]{Setting}

\newtheorem{defn}[thm]{Definition}
\newtheorem{quest}[thm]{Question}
\newtheorem{ex}[thm]{Example}

\Crefname{thm}{Theorem}{Theorems}
\Crefname{section}{Section}{Sections}
\Crefname{ex}{Example}{Examples}

\def\Kos{\mathcal{K}}
\def\K{\mathbb{K}}

\def\m{\mathfrak{m}}

\def\O{\mathcal{O}}

\def\P{\mathbb{P}}
\def\q{\mathfrak{q}}

\def\S{\mathcal{S}}

\def\Z{\mathcal{Z}}

\usepackage{bm}

\def\f{\bm{f}}
\def\y{\bm{y}}

\def\dim{\mathop{\rm dim}}

\def\im{\mathop{\rm im}}
\def\hgt{\mathop{\rm ht}}
\def\rk{\mathop{\rm rank}}

\def\Span{\mathop{\rm Span}}

\def\bideg{\mathop{\rm bideg}}
\def\syz{\mathop{\rm syz}}

\def\Res{\mathrm{Res}}
\def\Syl{\mathrm{Syl}}
\def\Hom{{\rm Hom}}

\DeclareRobustCommand{\longerrightarrow}{
\DOTSB\relbar\joinrel\relbar\joinrel\relbar\joinrel\relbar\joinrel\rightarrow
}
\DeclareRobustCommand{\longestrightarrow}{
\DOTSB\relbar\joinrel\relbar\joinrel\relbar\joinrel\relbar\joinrel\relbar\joinrel\relbar\joinrel\rightarrow
}

\usepackage[all]{xy}

\usepackage{nicematrix}

\begin{document}

\begin{abstract}
Let $U\subseteq H^0(\O_{\P^1\times \P^1}(a,b))$ be a four-dimensional vector space and consider the rational map $\phi_U:\,\P^1\times \P^1 \dashrightarrow \P^3$ defined by its basis of bihomogeneous polynomials. The tensor product surface $X_U\subseteq \P^3$ is the closed image of $\phi_U$, and a fundamental problem in this setting is to determine its implicit equation. As these surfaces are ubiquitous within the field of geometric modeling and design, knowledge of their implicit equations is particularly advantageous, allowing for more effective and efficient computations. In this article, we expand upon work of Duarte and Schenck \cite{DS16} and the present author \cite{Weaver25} to solve this implicitization problem when the bigraded ideal $I_U$ admits a singly graded syzygy.
\end{abstract}

\maketitle


\section{Introduction}

Given a map between projective spaces $\phi:\,\P^m\dashrightarrow\P^n$, a natural problem which arises is to describe the closed image of $\phi$ as a subvariety of $\P^n$. As such a map is defined by homogeneous polynomials in the coordinates of $\P^m$, one must obtain the equations of $\im \phi$ in the coordinates of $\P^n$. The \textit{implicitization problem}, determining the implicit equations of a parameterized variety, has become a fundamental question studied to great length by algebraic geometers and commutative algebraists, using a variety of techniques.

In recent years, the implicitization problem has gained significant interest within the geometric modeling community for its applications to computer-aided geometric design (CAGD), see e.g. \cite{CGZ00,SC95,SGD97,SSQK94}. Parameterized curves and surfaces are ubiquitous in this area, yet often present computational challenges. Many fundamental tasks may be computationally intensive given a parametric form, but become trivial given an implicit model. As such, the development of computationally effective implicitization methods has become an active area of research, drawing upon ideas from algebraic geometry, commutative algebra, and computational methods within both realms.

In this context, \textit{tensor product surfaces} form a notable class of parameterized surfaces. These surfaces arise as the image of a rational map $\P^1\times \P^1\dashrightarrow \P^3$, and their parameterizations inherit a rich structure from the bigraded coordinate ring of $\P^1\times \P^1$. This bigraded structure, however, tends to complicate the implicitization problem, introducing additional algebraic phenomena not witnessed in the singly graded setting. To address this, a variety of techniques have been developed for the implicitization of these surfaces. Common elimination techniques relying on Gr\"obner bases and resultants are straightforward to implement, but tend to be computationally expensive, especially in the multigraded setting. In contrast, syzygy techniques and methods borrowed from the study of \textit{Rees algebras} tend to be much more effective and are hence preferable, see e.g. \cite{Botbol11,BDD09,BC05,BJ03,WG18,Weaver26}. One such method in particular is the \textit{approximation complex} $\Z$ developed by Herzog, Simis, and Vasconcelos \cite{HSV82,HSV83}. Whereas this complex is rarely acyclic, it detects useful syzygetic information and has been shown to be particularly effective for implicitization purposes \cite{Botbol11, BC05, Chardin06,DS16,SSV14,Weaver25}. We refer the reader to \cite{Cox03} for a concise overview of syzygy methods used for implicitization and \cite{Chardin06} for an introduction to the approximation complex $\Z$ and its applications to the implicitization problem.

In this paper, we consider tensor product surfaces which are \textit{basepoint free}, i.e. surfaces arising when the rational map above is regular. Let $R=\K[s,t,u,v]$ be a polynomial ring, for $\K$ an algebraically closed field, bigraded by setting $\bideg s,t =(1,0)$ and $\bideg u,v = (0,1)$. Notice that the global sections $H^0(\O_{\P^1\times \P^1}(a,b))$ may be identified with the bigraded components $R_{a,b}$ of $R$. For a subspace $U\subseteq R_{a,b}$ with basis $\{p_0,p_1,p_2,p_3\}$, we assume that $p_0,p_1,p_2,p_3$ have no common zeros on $\P^1\times \P^1$. As such, the rational map 
$$\phi_U:\,\P^1\times \P^1 \longrightarrow \P^3$$
defined by $\{p_0,p_1,p_2,p_3\}$ is regular, as its base locus is empty, i.e. $U$ is \textit{basepoint free}. Writing $I_U=(p_0,p_1,p_2,p_3)$ to denote the ideal generated by these bihomogeneous forms, notice that $U$ is free of basepoints if and only if $\sqrt{I_U} = (s,t)\cap (u,v)$. In this setting, it is known from \cite{BDD09} that the implicit equation of $X_U$ may be obtained from the approximation complex $\Z$ on the generators of $I_U$. As $\Z$ is constructed from the Koszul complex on $\{p_0,p_1,p_2,p_3\}$ (see \Cref{Prelim Section} for details), an understanding of the homological properties of $I_U$ is particularly useful in this setting.

In \cite{SSV14}, tensor product surfaces $X_U$ for $U\subset R_{2,1}$ are studied in this manner by determining all possible free resolutions of $I_U$. It is noted that the presence of a \textit{linear} syzygy, in bidegree $(0,1)$ or $(1,0)$, on $I_U$  imposes certain constraints and determines some of the differentials in a minimal resolution. This observation is further developed in \cite{DS16}, where it is shown that the implicit equation of $X_U$ can be completely determined from such a linear syzygy, via the approximation complex, without requiring the full syzygy module of $I_U$. More recently, in \cite{Weaver25} these ideas were extended to the case where $I_U$ admits a \textit{quadratic} syzygy in bidegree $(0,2)$ or $(2,0)$. There it is shown that, similarly, such a syzygy determines a particular submodule of $\syz(I_U)$ that is sufficient to determine the implicit equation of $X_U$.

In this paper, we continue this line of work from \cite{DS16,Weaver25} and consider the general setting that the bigraded ideal $I_U$ admits a singly graded syzygy in any degree. We show that this syzygy $S$ alone determines a submodule of $\syz(I_U)$ sufficient to determine the implicit equation of the surface $X_U$, via the approximation complex $\Z$. The key construction used throughout is a subspace $V \subseteq U$ associated to the syzygy $S$, whose basis may be taken as part of a generating set of $I_U$. With this, there are a handful of cases to consider, depending on the dimension of $V$, with each case exhibiting unique behavior. The main results of this article, \Cref{dim V = 2 - main theorem,dim V = 3 - main theorem,dim V = 4 - main theorem}
are summarized as follows.

\begin{thm-nonum}
Assume that $U\subseteq H^0(\O_{\P^1\times \P^1}(a,b))$ is basepoint free, and let $I_U$ denote the ideal of $U$. Assume that $I_U$ has a minimal syzygy $S$ in bidegree $(0,n)$ and $b\geq 2n-1$. Let $V$ be the subspace of $U$ associated to the syzygy $S$.
\begin{enumerate}
    \item[(i)] If $\dim V=2$, then $I_U$ has syzygies $S_1$ and $S_2$ in bidegree $(a,b-n)$ such that $\dim \langle S,S_1,S_2\rangle_{2a-1,b-1}=2ab$.
    \vspace{1mm}

    \item[(ii)] If $\dim V=3$, then $I_U$ has syzygies $S_1$ in bidegree $(a,b-n)$, $S_2$ in bidegree $(a,b-n+\mu)$, and $S_3$ in bidegree $(a,b-\mu)$ for some $\mu$, such that $\dim \langle S,S_1,S_2,S_3\rangle_{2a-1,b-1}=2ab$.
 \vspace{1mm}
 
    \item[(iii)] If $\dim V=4$, then $I_U$ has syzygies $S_1$ in bidegree $(a,b-n+\mu_1)$, $S_2$ in bidegree $(a,b-n+\mu_2)$, and $S_3$ in bidegree $(a,b-\mu_1-\mu_2)$ for some $\mu_1, \mu_2$, such that $\dim \langle S,S_1,S_2,S_3\rangle_{2a-1,b-1}=2ab$.
\end{enumerate}
In each case, the first differential $d_1$ of $\Z_{2a-1,b-1}$ is determined by these syzygies. Moreover, the determinant of a $2ab\times 2ab$ matrix representation of $d_1$ is a power of the implicit equation of $X_U$.
\end{thm-nonum}

By symmetry, a similar result holds if $a\geq 2n-1$ and $I_U$ has a syzygy in bidegree $(n,0)$. Moreover, we note that the additional syzygies on $I_U$ are constructed explicitly and their description is formulaic, hence these constructions may be easily implemented into a computer algebra system, such as \texttt{Macaulay2} \cite{Macaulay2}. In particular, the methods presented here yield more efficient computation of the implicit equation of $X_U$ as they do not require computation of the entire syzygy module $\syz(I_U)$. Lastly, we note that this result recovers the main result of \cite{DS16} when $n=1$, noting $\dim V=2$ in this case, and also recovers the main results of \cite{Weaver25} when $n=2$, noting that $\dim V=2$ or $\dim V=3$ (in which case $\mu =1$) in this setting.

We briefly describe how this paper is organized. In \Cref{Prelim Section}, we introduce the necessary preliminary material required for this article. We review the construction of the approximation complex $\Z$ and its relation to syzygies, and recall the techniques of \cite{Botbol11} using this complex for implicitization. In \Cref{Graded Syz Section}, we introduce the main setting of the paper for a bigraded tensor product surface admitting a singly graded syzygy $S$. The subspace $V\subseteq U$ associated to $S$ is introduced, and it is shown that there are three cases to consider, depending on its dimension. In Sections \ref{dim 2 section}, \ref{dim 3 section}, and \ref{dim 4 section}, we consider the cases that $\dim V =2,3,4$ and construct the additional syzygies above, noting that a basis of $V$ may be taken as part of a generating set of $I_U$. In \Cref{dim 2 section}, with this chosen generating set, the key observation is that the nonzero entries of $S$ form a regular sequence. Similarly, in \Cref{dim 3 section} we use that the nonzero entries of $S$ generate a perfect ideal of codimension two, allowing for the use of the Hilbert-Burch theorem. Lastly, in \Cref{dim 4 section} we use that, when $\dim V=4$, we have $V=U$ and hence a basis of $V$ generates $I_U$. In \Cref{Open questions section}, we conclude the article with several open questions and observations stemming from the results presented here.


\section{Preliminaries}\label{Prelim Section}

We briefly review the preliminary material needed for this paper. We begin with the construction of the \textit{approximation complex} $\Z$ \cite{HSV82,HSV83} associated to the ideal $I_U$. We then recall its applications to the implicitization of tensor product surfaces, following \cite{Botbol11}.

\subsection{Approximation Complex}

We recall the construction of the \textit{approximation complex} $\Z$, introduced by Herzog, Simis, and Vasconcelos in \cite{HSV82,HSV83}. Whereas this complex may be defined more generally, we consider the setting of a rational map $\P^m\dashrightarrow \P^n$ to better match the framework used for the duration of the article.

For such a rational map, consider the ideal of homogeneous forms $I=(f_0,\ldots,f_n) \subseteq R=\K[x_0,\ldots,x_m]$ defining it. Write $\Kos(\f)$
to denote the Koszul complex on the generating set $\f=f_0,\ldots,f_n$ with differentials $d_i^{\f}$. Writing $S=\K[y_0,\ldots,y_n]$ to denote the coordinate ring of $\P^n$, consider the Koszul complex $\Kos(\y)$ on the sequence $\y= y_0,\ldots,y_n$ with differentials $d_i^{\y}$. We now introduce the approximation complex $\Z$ as a hybrid complex from this data.

\begin{defn}
Writing $Z_i = \ker d_i^f$ to denote the cycles of $\Kos(\f)$, the approximation complex $\Z$ is the complex 
$$\Z\,:\,\,0 \rightarrow \Z_{n+1}\overset{d_{n+1}}{\longrightarrow}\Z_{n}\overset{d_{n}}{\longrightarrow} \cdots\cdots \longrightarrow\Z_{1}\overset{d_{1}}{\longrightarrow} \Z_0 $$
where $\Z_i = S\otimes_\K Z_i$ and $d_i = d_i^{\y}$.
\end{defn}

A direct computation shows that $d_{i-1}^{\f} d_{i}^{\y}+d_{i-1}^{\y} d_{i}^{\f} =0$. Hence for any $g\in \Z_i$, we see that $d_{i-1}^{\f} d_{i}^{\y} (g) =-d_{i-1}^{\y} d_{i}^{\f}(g) =0$, hence $d_i(g) \in \Z_{i-1}$ and so these maps are well defined. Moreover, it is clear that $\Z$ is a complex as its differentials agree with those of $\Kos(\y)$. Lastly, we note that the approximation complex $\Z$ depends only on the ideal $I$, and not the choice of generating set.

Now that $\Z$ has been shown to be a complex of modules over $S\otimes_\K R$, we note that it is rarely acyclic. However, within the study of Rees algebras, the homology of $\Z$ is often used to measure the deviation between the Rees ring of $I$ and the \textit{symmetric algebra} of $I$. Indeed, notice that $Z_0=R$ and $Z_1=\syz(I)$, the module of syzygies of $I$. Hence it follows that the first differential of $\Z$ maps $d_1\,:S\otimes_\K\syz(I) \rightarrow S\otimes_\K R$ by
\begin{equation}\label{d_1 syz map}
(a_0,\ldots,a_n)\longmapsto a_0y_0+\cdots+a_ny_n.
\end{equation}
for $(a_0,\ldots,a_n)\in \syz(I)$. With this, it follows that the zeroth homology of $\Z$ is $H_0(\Z) \cong \S(I)$, where $\S(I)$ is the symmetric algebra of $I$.

\begin{rem}\label{syzygies to approx comp}
With the description of the first differential of $\Z$ in (\ref{d_1 syz map}), we note that a representation of $d_1$ can be obtained from the first syzygy matrix of $I$. Indeed, if $R^p\overset{\varphi}{\rightarrow} R^{n+1} \rightarrow I\rightarrow 0$ is a free presentation with syzygy matrix $\varphi$, then $\im d_1 = (\ell_1,\ldots,\ell_p) \subseteq R[y_0,\ldots,y_n] \cong S\otimes_\K R$, where $[\ell_1\ldots \ell_p] = [y_0\ldots y_n]\cdot \varphi$.
\end{rem}

Following \Cref{syzygies to approx comp}, it suffices to understand the syzygies of $I$ to determine this differential. This observation will be particularly useful in the implicitization of tensor product surfaces.

\subsection{Multigraded Implicitization}

We now recall the applications of the approximation complex $\Z$ to the implicitization of tensor product surfaces. This complex has been used in multiple instances for implicitization \cite{BDD09,BC05,Chardin06}, however we refer to the methods developed in \cite{Botbol11} for their applications to the multigraded setting. Notice that $S\otimes_\K R$ is naturally bigraded. Hence 
for $\nu$ a fixed degree in the grading of $R$, the strand 
$$\Z_\nu\,: \quad \cdots \longrightarrow S\otimes_\K (Z_i)_\nu \overset{d_i}{\longrightarrow} S\otimes_\K (Z_{i-1})_\nu \overset{d_{i-1}}{\longrightarrow} \cdots.$$
is a graded complex of $S$-modules. If $R$ is bigraded, as in the setting of a tensor product surface parameterization, one may take a \textit{bigraded} strand $\Z_\nu$, for $\nu$ in the bigrading of $R$, of $S$-modules as well.

\begin{lemma}[{\cite[7.3]{Botbol11}}]\label{Botbol det Z}
Let $U=\{p_0,p_1,p_2,p_3\} \subseteq R_{a,b}$ and let $\phi_U:\,\P^1\times\P^1 \dashrightarrow \P^3$ denote the rational map defined by $U$. Assume that either $\phi_U$ has no basepoints or $\phi_U$ has finitely many basepoints that are locally a complete intersection. Let $\nu=(2a-1,b-1)$ (equivalently $\nu =(a-1,2b-1)$) and let $\Delta_\nu= \det \Z_\nu$. Then
$$\deg (\Delta_\nu) = 2ab - \dim (H_2)_{4a-1,3b-1}.$$
Moreover, the differential $d_1:\, (\Z_1)_\nu \rightarrow (\Z_0)_\nu$ has a matrix representation which is square of size $2ab\times 2ab$ if and only if $(H_2)_{4a-1,3b-1}=0$.
\end{lemma}

Here $(H_2)_{4a-1,3b-1}$ denotes the second homology module of the bigraded strand $\Z_{4a-1,3b-1}$. For our purposes we will only consider the setting where $U$ is free of basepoints, in which case $(H_2)_{4a-1,3b-1}=0$ and the determinant of $\Z_{2a-1,b-1}$ is simply the determinant of a square matrix representation of $d_1$ in bidegree $(2a-1,b-1)$. However, we refer the reader to \cite{Chardin06} for the general notion of the determinant of a complex, within the context of implicitization.

\begin{lemma}[{\cite[7.4]{Botbol11}}]\label{Botbol degree lemma}
With the conditions of \Cref{Botbol det Z}, suppose that the basepoints of $U$ (if any) have multiplicity $e_x$. One has
$$\deg (\phi_U)\deg(F) = 2ab -\sum e_x$$
where $F\in S=\K[y_0,y_1,y_2,y_3]$ is the implicit equation of $X_U$.
\end{lemma}

Combining \Cref{Botbol det Z,Botbol degree lemma}, we obtain our primary tool to determine the implicit equation of $X_U$, which will be used throughout the article.

\begin{thm}[{\cite[7.5]{Botbol11}}]\label{Botbol implicit eqn}
    With the assumptions of \Cref{Botbol det Z}, we have that $\Delta_\nu = F^{\deg \phi_U}$. In particular, from  \Cref{Botbol degree lemma} we have
    $$\deg \Delta_\nu =\deg (F^{\deg \phi_U}) = (\deg F)(\deg \phi_U) =  2ab -\sum e_x.$$
    Hence by \Cref{Botbol det Z}, it follows that $\dim (H_2)_{4a-1,3b-1} =\sum e_x$.
\end{thm}

In the absence of basepoints, the differential $(d_1)_\nu:\, (\Z_1)_\nu \rightarrow (\Z_0)_\nu$ is a square $2ab\times 2ab$ matrix whose determinant is a power of the implicit equation $F$. Whereas the differential $d_1:\,\Z_1\rightarrow \Z_0$ is determined by the syzygy module of $I_U$ by \Cref{syzygies to approx comp}, we note that the bigraded strand $(d_1)_\nu:\, (\Z_1)_\nu \rightarrow (\Z_0)_\nu$ is often determined by far fewer syzygies of $I_U$.


\section{Graded syzygies of bigraded ideals}\label{Graded Syz Section}

With the necessary preliminaries established in \Cref{Prelim Section}, we now introduce the main setting for the duration of the paper. Although many of the conventions have been stated in the introduction, we briefly restate them here for convenience. Our primary setting is the following.

\begin{set}\label{General Setting}
Let $R=\K[s,t,u,v]$ with $\bideg s,t =(1,0)$ and $\bideg u,v = (0,1)$. Let $U \subseteq R_{a,b}$ be a subspace with basis $\{p_0,p_1,p_2,p_3\}$ and let $I_U = (p_0,p_1,p_2,p_3) \subseteq R$. Assume that $U$ is basepoint free and let $\phi_U:\, \P^1\times \P^1 \longrightarrow \P^3$ be the regular map defined by $U$, with image $X_U$. Assume that $I_U$ has a first syzygy $S$ of bidegree $(0,n)$ and that $b\geq 2n-1$. Moreover, assume that $S$ is a singly graded syzygy in minimal degree, i.e. that $I_U$ has no syzygy in bidegree $(0,m)$ for $m<n$.
\end{set}

Here we may safely assume that $S$ has minimal degree, in case there are multiple singly graded syzygies. Moreover, this minimality condition will be particularly useful in the following arguments.

To begin, we follow the approach of \cite{Weaver25} and construct a $\K$-vector space associated to the syzygy $S$. As $S$ consists of homogeneous polynomial entries in the subring $A=\K[u,v]$, we have
$$\sum_{i=0}^3 \big(\sum_{j=0}^n a_{ij} u^{n-j}v^j\big) p_i =0$$
for some $a_{ij}\in \K$. Expanding and rearranging, we obtain
\begin{equation}\label{Syzygy rearrangement}
    0=\sum_{i=0}^3 \big(\sum_{j=0}^n a_{ij} u^{n-j}v^j\big) p_i = \sum_{j=0}^n \big( \sum_{i=0}^3 a_{ij} p_i\big) u^{n-j}v^j= \sum_{j=0}^n f_j u^{n-j}v^j
\end{equation}
where we write $f_j = \sum_{i=0}^3 a_{ij} p_i$ for $0\leq j\leq n$.

Similar to the observations made in \cite{Weaver25}, the collection $\{f_0,\ldots,f_n\}$ will be key to determining the implicit equation of $X_U$. Notice that, as these polynomials are combinations of the generators of $I_U$, each $f_j$ is bihomogeneous of bidegree $(a,b)$. In particular, the set $\{f_0,\ldots,f_n\}$ spans a $\K$-subspace of $U$.

\begin{rem}\label{f0 and fn nonzero}
We note that the set $\{f_0,\ldots,f_n\}$ is non-vanishing. Indeed, since $\{p_0,p_1,p_2,p_3\}$ is a basis of $U$, we note that $f_j =0$ if and only if $a_{0j} = a_{1j} = a_{2j} = a_{3j}=0$. As these are the coefficients of the polynomial entries of $S$, it follows that at least one of $f_0,\ldots,f_n$ is nonzero. In particular, from the assumptions above, it follows that both $f_0 \neq 0$ and $f_n \neq 0$. Indeed, if $f_0 =0$, then (\ref{Syzygy rearrangement}) shows that 
$$0=\sum_{j=1}^n f_j u^{n-j}v^j = v\cdot\sum_{j=1}^n f_j u^{n-j}v^{j-1}.$$
However, since $R$ is a domain, we then have $\sum_{j=1}^n f_j u^{n-j}v^{j-1} =0$. Hence by rearranging as in (\ref{Syzygy rearrangement}), it follows that $I_U$ has a syzygy in bidegree $(0,n-1)$, which is a contradiction. Similarly, it follows that $f_n\neq 0$.
\end{rem}

Following the approach of \cite{Weaver25}, we consider the vector space $V = \Span_\K\{f_0,\ldots,f_n\}$. Note that this is a $\K$-subspace of $U =  \Span_\K\{p_0,p_1,p_2,p_3\}$.

\begin{prop}\label{Dimension of V}
With the assumptions of \Cref{General Setting} and $V$ the subspace spanned by $\{f_0,\ldots,f_n\}$, we have that $2\leq \dim V\leq 4$.
\end{prop}

\begin{proof}
The latter inequality is clear, noting that $V$ is a subspace of $U$. Following \Cref{f0 and fn nonzero}, we have that $V \neq 0$, and so it suffices to show that $\dim V \neq 1$. Suppose that $\dim V =1$ and recall that $f_0 \neq 0$ by \Cref{f0 and fn nonzero}. Thus $\{f_0\}$ is a basis for $V$, and so we have $f_j = b_jf_0$ for some $b_j \in \K$, for $1\leq j\leq n$. Thus by (\ref{Syzygy rearrangement}), we see
$$0=\sum_{j=0}^n f_j u^{n-j}v^j = f_0 \cdot \sum_{j=0}^n b_j u^{n-j}v^j$$
where we take $b_0 =1$. Since $R$ is a domain and $f_0 \neq 0$, we have $\sum_{j=0}^n b_j u^{n-j}v^j=0$. However, this is impossible as $\{u^n,u^{n-1}v,\ldots,v^n\}$ is $\K$-linearly independent and $b_0\neq 0$. 
\end{proof}

By \Cref{Dimension of V}, there are three possible values for $\dim V$. In the following sections, we consider the implicitization of $X_U$ from the syzygies of $I_U$ in each of these cases.

\begin{rem}\label{low n dimension remark}
We note the range of $\dim V$ in \Cref{Dimension of V} can be refined when $n$ is small. Indeed, since $V$ is spanned by $n+1$ polynomials, we have $2\leq \dim V\leq \min\{n+1,4\}$. In particular, when $n=1$ and $I_U$ has a linear syzygy, as in the setting of \cite{DS16}, then $\dim V=2$. If $n=2$, then $\dim V=2$ or $\dim V=3$, as noted in \cite{Weaver25}. As these cases have been previously studied, one may safely assume that $n\geq 3$ throughout. However, we do not impose this restriction, so as to extend and recover these previous results.
\end{rem}

\begin{rem}\label{first dim V generators}
We note that it is relatively simple to compute $\dim V$ from the assumptions in \Cref{General Setting}. With the coefficients of the polynomials in S in (\ref{Syzygy rearrangement}), write $A=(a_{ij})$ to denote the $4\times (n+1)$ transition matrix with
\begin{equation}\label{A transition matrix equation}
[f_0\ldots f_n] = [p_0\ldots p_3]\, A.
\end{equation}
Since $\{p_0,p_1,p_2,p_3\}$ is a basis of $U$, and hence $\K$-linearly independent, it follows that $\dim V =\rk A$, which is a straightforward linear algebra computation. With this, and the transition equation above, we also note that a basis of $V$ may be taken as part of a minimal generating set of $I_U$. Indeed, from the matrix equation above, this follows as $\dim V =\rk A$ is the largest size of a nonzero, and hence invertible, minor of $A$.
\end{rem}

To avoid cumbersome notation with the indices, write $\{f_0',\ldots,f_{\dim V-1}'\}$ to denote a chosen subset of $\{f_0,\ldots,f_n\}$ which forms a basis of $V$. As mentioned, these basis polynomials may be taken as minimal generators of $I_U$, hence after possibly rearranging we may take $f_0',\ldots,f_{\dim V-1}'$ as the first $\dim V$ generators of $I_U$. In particular, any syzygy on the subideal $(f_0',\ldots,f_{\dim V-1}')$ may be extended to a syzygy on $I_U$. 

\begin{rem}\label{B transition matrix and syzygies remark}
With $\{f_0',\ldots,f_{\dim V-1}'\}$ a basis of $V$, there exists a $(\dim V) \times(n+1)$ transition matrix $B$ with entries in $\K$ such that 
\begin{equation}\label{B transition matrix equation}
[f_0\ldots f_n] = [f_0'\ldots f_{ \dim V -1}']\, B. 
\end{equation}
With this, notice that if $C$ is a syzygy on $(f_0,\ldots,f_n)$, then $BC$ is a syzygy on $(f_0',\ldots, f_{\dim V -1}')$. In particular, from (\ref{Syzygy rearrangement}) we see that $[u^n\, u^{n-1}v\, \ldots \,v^n]^T$ is a syzygy on $(f_0,\ldots,f_n)$. Thus $B [u^n\,\, u^{n-1}v\,\, \ldots \,\,v^n]^T$ is a syzygy on $(f_0',\ldots, f_{\dim V -1}')$. Moreover, choosing $\{f_0',\ldots, f_{\dim V -1}'\}$ as the first $\dim V$ generators of $I_U$ as in \Cref{first dim V generators}, this syzygy is precisely the singly graded syzygy $S$ in \Cref{General Setting}.
\end{rem}


\section{Syzygies in the case $\dim V=2$}\label{dim 2 section}

We begin with the first case of \Cref{Dimension of V} and assume that $\dim V=2$ throughout. With the syzygy $S$ in bidegree $(0,n)$, we note this setting has been studied for $n=1$ in \cite{DS16} and for $n=2$ in \cite{Weaver25}, following the observation made in \Cref{low n dimension remark}.

With the assumptions of \Cref{General Setting} and the singly graded syzygy $S$, we may construct the set $\{f_0,\ldots,f_n\}$ as in (\ref{Syzygy rearrangement}). As we assume that $\dim V=2$, there are $f_0', f_1'$ among this set such that $\{f_0',f_1'\}$ is a basis for $V$. Following \Cref{first dim V generators}, after possibly rearranging, we take $I_U = (f_0',f_1',p_2,p_3)$ as our chosen generating set for the duration of this section. As our goal is to obtain syzygies on $I_U$ sufficient to produce the implicit equation, we begin by describing the syzygetic behavior of the subideal generated by $f_0',f_1'$.

By \Cref{B transition matrix and syzygies remark}, there exists a $2\times (n+1)$ transition matrix $B$ with entries in $\K$ such that  
\begin{equation}\label{dim V = 2 - basis transition equation}
[f_0'\, \,f_1']\, B= [f_0\,\ldots\,f_n].
\end{equation}
Furthermore, recall from (\ref{Syzygy rearrangement}) that $[u^n\,\, u^{n-1}v\,\, \ldots \,\,v^n]^T$ is a syzygy on $(f_0,\ldots,f_n)$, hence  
\begin{equation}\label{dim V=2 - g_0,g_1 product equation}
    B\, [u^n\, u^{n-1}v\, \ldots \,v^n]^T = \begin{bmatrix}
    g_0\\
    g_1
\end{bmatrix}
\end{equation}
is a syzygy on $(f_0',f_1')$, for some $g_0,g_1\in \K[u,v]$. Moreover, from \Cref{B transition matrix and syzygies remark} and the chosen generating set $I_U = (f_0',f_1',p_2,p_3)$, the syzygy $S$ in bidegree $(0,n)$ in \Cref{General Setting} is precisely
\begin{equation}\label{dim V =2 - S defn}
S=\begin{bmatrix}
    g_0\\
    g_1\\
    0\\
    0
\end{bmatrix}_.
\end{equation}

\begin{prop}\label{dim V = 2 - g_0 g_1 regular sequence}
With the assumptions of \Cref{General Setting}, assume that $\dim V =2$. With (\ref{dim V=2 - g_0,g_1 product equation}), we have that $g_0, g_1$ is an $R$-regular sequence.    
\end{prop}

\begin{proof}
Since $\dim V =2$ and $\{f_0',f_1'\}$ is a basis, it follows from (\ref{dim V = 2 - basis transition equation}) that $B$ has rank 2. With this, and since $\{u^n,u^{n-1}v,\ldots,v^n\}$ is $\K$-linearly independent, it follows from (\ref{dim V=2 - g_0,g_1 product equation}) that both $g_0$ and $g_1$ are nonzero. Thus $g_0, g_1$ is a regular sequence if and only if $g_0$ and $g_1$ have no nontrivial common factor. However, if $g_0$ and $g_1$ do share such a factor, it may be factored from $S$ in (\ref{dim V =2 - S defn}) to produce a singly graded syzygy of strictly smaller degree, which contradicts the assumptions of \Cref{General Setting}.    
\end{proof}

Notice that, since $[g_0\,\,g_1]^T$ is a syzygy on $(f_0',f_1')$, it follows that $[f_0'\,\,f_1']^T$ is a syzygy on $(g_0,g_1)$. However, by \Cref{dim V = 2 - g_0 g_1 regular sequence}, this ideal is generated by a regular sequence, hence $[f_0'\,\,f_1']^T$ must be a multiple of its Koszul syzygy. Thus we have
\begin{equation}\label{dim V=2 - f_0' f_1' matrix equation}
    \left[
\begin{array}{c}
     f_0' \\
     f_1'
\end{array}
    \right] = \left[
\begin{array}{c}
     g_1 \\
     -g_0
\end{array}
    \right]\alpha
\end{equation}
for some bihomogeneous $\alpha\in R_{a,b-n}$. In particular, the syzygy $S$ in (\ref{dim V =2 - S defn}) is a \textit{reduced} Koszul syzygy on $I_U$.

Before we construct the necessary additional syzygies on $I_U$, we state a brief containment lemma for homogeneous ideals in two variables. Whereas this is likely well known to experts, we offer a brief proof using the classical Sylvester resultant in the homogeneous setting.

\begin{lemma}\label{Sylvester containment lemma}
Let $f$ and $g$ be homogeneous polynomials in $\K[x,y]$ with $\deg f = m$ and $\deg g =n$. If $\gcd (f,g) = 1$, then there is a containment $(x,y)^{m+n-1} \subseteq (f,g)$.
\end{lemma}

\begin{proof}
Writing $f= \sum_{i=0}^m a_ix^{m-i}y^i$ and $g = \sum_{i=0}^n b_ix^{n-i}y^i$ for $a_i,b_i \in \K$, we may form the $(m+n)\times (m+n)$ homogeneous \textit{Sylvester matrix} of $f$ and $g$, with respect to the indeterminates $x,y$:
\[
\arraycolsep=2pt\def\arraystretch{0.8}
\Syl_{xy}(f,g)=
\left[
\begin{array}{ccccccc}
a_0 & a_{1} & \cdots & a_m &        &        &        \\
    & a_0     & a_{1}& \cdots & a_m &        &        \\
    &         & \ddots &        &     & \ddots &        \\
    &         &        & a_0    & a_{1} & \cdots & a_m \\[1ex]
b_0 & b_{1} & \cdots & b_n &        &        &        \\
    & b_0     & b_{1}& \cdots & b_n &        &        \\
    &         & \ddots &        &     & \ddots &        \\
    &         &        & b_0    & b_{1} & \cdots & b_n
\end{array}
\right]
\begin{array}{l}
\left.\rule{0pt}{0.95cm}\right\}\;\text{$n$}\\[6pt]
\left.\rule{0pt}{0.95cm}\right\}\;\text{$m$}
\end{array}
\]
With this, notice that 
\[
\arraycolsep=2pt\def\arraystretch{0.8}
\Syl_{xy}(f,g) \left[\begin{array}{c}
    x^{m+n-1}\\
    x^{m+n-2}y\\
    \vdots \\[1ex]
    y^{m+n-1}
\end{array}\right]= \left[\begin{array}{c}
    x^{n-1}f\\
    x^{n-2}y\,f\\
    \vdots\\
    y^{n-1}f\\[1ex]
    \hline\\[-2ex]
    x^{m-1}g\\
    x^{m-2}y\,g\\
    \vdots\\
    y^{m-1}g
\end{array}
\right]
\]
and recall that, since $\gcd (f,g) =1$, the \textit{resultant} of these polynomials, $\Res_{xy}(f,g) =  \det \Syl_{xy}(f,g)$, is nonzero. As $\det \Syl_{xy}(f,g) \neq 0$, this matrix is invertible as it consists of entries in $\K$. Multiplying by this inverse in the matrix equation above shows that 
$$(x,y)^{m+n-1} \subseteq (f)(x,y)^{n-1}+(g)(x,y)^{m-1} \subseteq (f,g)$$
as claimed.
\end{proof}

\begin{rem}\label{power max ideal in (g0,g1)}
Notice that by \Cref{dim V = 2 - g_0 g_1 regular sequence} and \Cref{Sylvester containment lemma} we have $(u,v)^{2n-1}\subseteq (g_0,g_1)$, which will be essential to the proof of this section's main result. We note that if $n=1$, the transition matrix $B$ in (\ref{dim V=2 - g_0,g_1 product equation}) is the identity matrix.  Hence $g_0 =u$ and $g_1=v$, and so the containment $(u,v)^{2n-1}\subseteq (g_0,g_1)$ is clear in this case and the setting of \cite{DS16}. For the case $n=2$, the containment $(u,v)^{2n-1}\subseteq (g_0,g_1)$ was made explicit in \cite{Weaver25}, depending on the choice of basis. In particular, the equations in \cite[4.1]{Weaver25} are precisely the result from multiplying by the inverse of $\Syl_{uv}(g_0,g_1)$.
\end{rem}

We now construct an additional pair of syzygies which, with $S$ in (\ref{dim V =2 - S defn}), will determine the differential in the approximation complex $\Z_{2a-1,b-1}$, and hence yield the implicit equation following \Cref{Botbol implicit eqn}.

\begin{thm}\label{dim V = 2 - main theorem}
With the assumptions of \Cref{General Setting}, assume that $\dim V =2$. The ideal $I_U$ has two additional syzygies $S_1, S_2$ of bidegree $(a,b-n)$ such that $\dim \langle S,S_1,S_2\rangle_{2a-1,b-1}=2ab$. 
\end{thm}

\begin{proof}
We follow the approach taken in the proofs of \cite[2.2]{DS16} and \cite[4.3]{Weaver25}. Recall that $I_U = (f_0',f_1',p_2,p_3)$ following \Cref{first dim V generators}, and note that $f_0'=\alpha g_1$ and $f_1'=-\alpha g_0$ by (\ref{dim V=2 - f_0' f_1' matrix equation}). As noted, by \Cref{Sylvester containment lemma} it follows that $(u,v)^{2n-1}\subseteq (g_0,g_1)$, hence $p_2,p_3 \in (g_0,g_1)$ as $I_U$ is generated in bidegree $(a,b)$ and $b\geq 2n-1$. Thus we may write
\begin{equation}\label{p2,p3 decomp}
\left\{\begin{array}{rl}
     p_2 &= q_1g_1-q_0g_0  \\
     p_3 &= r_1g_1-r_0g_0 
\end{array}
\right.
\end{equation}
for some $q_0,q_1,r_0,r_1 \in R_{a,b-n}$. With this, notice that $q_1 f_0'+q_0f_1' - \alpha p_2 =0$, and also $r_1 f_0'+r_0f_1' -\alpha p_3 =0$. Hence both
\begin{equation}\label{dim V = 2 - defn of S1 and S2}
S_1=\begin{bmatrix}
    q_1\\
    q_0\\
-\alpha\\
    0
\end{bmatrix}\qquad\text{and} \qquad S_2=\begin{bmatrix}
    r_1\\
    r_0\\
0\\
    -\alpha
\end{bmatrix}
\end{equation}
are syzygies of $I_U$, with entries in bidegree $(a,b-n)$. 

Combining the syzygies $S$, $S_1$, and $S_2$, the syzygy module of $I_U$ contains the image of
\begin{equation}\label{Dim 2 M syzygy matrix}
M=\begin{bmatrix}
g_0 & q_1 & r_1\\
g_1 & q_0 & r_0\\
0 & -\alpha & 0\\
0 & 0 &-\alpha
\end{bmatrix}_.
\end{equation}
As the submatrix of $M$ consisting of its last three rows is upper triangular, the columns of $M$ span a free $R$-module. As these columns are precisely $S$, $S_1$, and $S_2$, the claim will follow once it has been shown that $M_{2a-1,b-1}$ consists of $2ab$ linearly independent columns, where the independence follows from this observation.

As the bidegrees of the entries of $S$, $S_1$, and $S_2$ are known, we need only count monomials in certain bidegrees to determine how many columns each syzygy contributes to $M_{2a-1,b-1}$. We also note that the number of columns contributed to $M_{2a-1,b-1}$ and a matrix representation of $d_1$ in $\Z_{2a-1,b-1}$ by each syzygy agree. As $S$ has entries in bidegree $(0,n)$, it yields
$$h^0(\O_{\P^1\times \P^1}(2a-1,b-n-1)) = 2a(b-n)$$
columns of $M_{2a-1,b-1}$. Similarly, as both $S_1$ and $S_2$ are syzygies of bidegree $(a,b-n)$, they each give rise to
$$h^0(\O_{\P^1\times \P^1}(a-1,n-1)) = an$$
columns of $M_{2a-1,b-1}$. Summing these figures, we see that $\dim \langle S,S_1,S_2\rangle_{2a-1,b-1} = 2ab$, again noting the independence follows as $\{S,S_1,S_2\}$ spans a free $R$-module.
\end{proof}

\begin{cor}
With the assumptions of \Cref{dim V = 2 - main theorem}, the first differential $d_1$ of the bigraded strand $\Z_{2a-1,b-1}$ of the approximation complex $\Z$ on the generators of $I_U$ is determined by the syzygies $\{S,S_1,S_2\}$. 
\end{cor}

\begin{proof}
This follows from \Cref{dim V = 2 - main theorem}, \Cref{Botbol det Z}, and \Cref{Botbol implicit eqn}, noting that $\dim \langle S,S_1,S_2\rangle_{2a-1,b-1}$ is precisely the number of columns contributed by these syzygies to a matrix representation of $d_1$.
\end{proof}

In particular, following \Cref{Botbol implicit eqn}, the determinant of the $2ab\times 2ab$ matrix representation of this differential is a power of the implicit equation of $X_U$. We refer the reader to examples \cite[1.3]{DS16} and \cite[4.5]{Weaver25} for a detailed illustration of this process, in the case that $n=1,2$. Moreover, we also offer a complete example, in the setting that $\dim V=4$ at the end of \Cref{dim 4 section}.

\begin{rem}
 We note that \Cref{dim V = 2 - main theorem} agrees with \cite[2.2]{DS16} and \cite[4.3]{Weaver25} if $n=1,2$. It is also worth mentioning that, as stated, the conjecture posed in \cite[6.4]{Weaver25} is not quite correct. Indeed, one requires that $b\geq 2n-1$ (and not $b\geq n+1$) to have the containment $(u,v)^{2n-1}\subseteq (g_0,g_1)$, which was essential to the proof above. Moreover, we note that when $n=1$, it is possible to have $b=1$, differing slightly from the setting of \cite{DS16}. Interestingly, when $n=b=1$, the syzygy $S$ is not necessary to describe the differential of $\Z_{2a-1,b-1}= \Z_{2a-1,0}$, and the two syzygies $S_1,S_2$ in bidegree $(a,0)$ are sufficient. 
\end{rem}

The case that $\dim V=2$ leads to relatively simple syzygies on $I_U$. Indeed, as noted, $S$ may be taken to be a reduced Koszul syzygy and the construction of $S_1$ and $S_2$ is relatively straightforward. Much of this is due to the fact that $g_0, g_1$ is a regular sequence, following \Cref{dim V = 2 - g_0 g_1 regular sequence}. Whereas this will fail to be the case in the following sections, we will continue to use the syzygetic behavior of the ideal of entries of $S$ when $\dim V=3,4$.

We also note that we employ a slightly different sign convention here, compared to the previous work of \cite{DS16,Weaver25}, the purpose of which will become more apparent in the following sections. Notice that the equations of (\ref{p2,p3 decomp}) may be written as the matrix products
\[
p_2= [g_1\,\, -g_0] \begin{bmatrix}
    q_1\\q_0
\end{bmatrix}_, \quad\quad\quad 
p_3= [g_1\,\, -g_0] \begin{bmatrix}
    r_1\\r_0
\end{bmatrix}_,
\]
involving the transposed syzygy in (\ref{dim V=2 - f_0' f_1' matrix equation}). Moreover, we note that these columns appear in the syzygies $S_1$ and $S_2$ in (\ref{dim V = 2 - defn of S1 and S2}). While not imperative in this setting, it will be necessary to express certain polynomials as matrix products involving syzygies when $\dim V$ is larger, in the following sections.


\section{Syzygies in the case $\dim V=3$}\label{dim 3 section}

We now consider the second case of \Cref{Dimension of V} when the vector space $V$ associated to the syzygy $S$ in \Cref{General Setting} has dimension $3$. We note this setting was briefly studied for $n=2$ in \cite{Weaver25}. Moreover, following \Cref{low n dimension remark}, $\dim V =3$ only if $n\geq 2$. In particular, we aim to recover and extend this previous work.

With the assumptions of \Cref{General Setting} and the syzygy $S$, we may form the set $\{f_0,\ldots,f_n\}$ as in (\ref{Syzygy rearrangement}). As $\dim V=3$, there is a subset $\{f_0',f_1',f_2'\}$ among this set forming a basis for $V$. By \Cref{first dim V generators}, and after possibly rearranging, we may assume that $I_U = (f_0',f_1',f_2',p_3)$. As noted in \Cref{B transition matrix and syzygies remark}, there exists a $3\times (n+1)$ transition matrix $B$ with entries in $\K$ such that  
\begin{equation}\label{dim 3 transition matrix equation}
[f_0'\, \,f_1'\,\,f_2']\, B= [f_0\,\ldots\,f_n].
\end{equation}
Additionally, recall from (\ref{Syzygy rearrangement}) that  $[u^n\, u^{n-1}v\, \ldots \,v^n]^T$ is a syzygy on $(f_0,\ldots,f_n)$, hence
\begin{equation}\label{dim 3 - syzygy on subideal}
B\, [u^n\, u^{n-1}v\, \ldots \,v^n]^T = \begin{bmatrix}
    g_0\\
    g_1\\
    g_2
\end{bmatrix}    
\end{equation}
 is a syzygy on $(f_0',f_1',f_2')$, for some $g_0,g_1,g_2 \in \K[u,v]$. Thus by \Cref{B transition matrix and syzygies remark} and writing $I_U = (f_0',f_1',f_2',p_3)$, the syzygy $S$ in bidegree $(0,n)$ in \Cref{General Setting} is precisely
\begin{equation}\label{dim 3 -- S defn}
    S=\begin{bmatrix}
    g_0\\
    g_1\\
    g_2\\
    0
\end{bmatrix}_.    
\end{equation}

Since the syzygy in (\ref{dim 3 - syzygy on subideal}) consists of three polynomials in $\K[u,v]$, it is impossible for these entries to form a regular sequence, as in the previous section. However, this ideal of entries does admit a prescribed syzygetic structure.

\begin{prop}\label{dim V = 3 - (g_0 g_1 g_2) Artinian so HB G2P}
    With the assumptions of \Cref{General Setting}, assume that $\dim V =3$. With (\ref{dim 3 - syzygy on subideal}), the ideal of entries $(g_0, g_1, g_2)$ is an Artinian ideal of $A=\K[u,v]$ with minimal free resolution
    \begin{equation}\label{dim 3 - res of (g0,g1,g2)}
 0 \longrightarrow \begin{array}{c}
 A(-n-\mu)\\
 \oplus \\
A(-2n+\mu)
\end{array} \overset{\psi}{\longrightarrow} \,
A^3(-n)\overset{[g_0\,\,g_1\,\,g_2]}{\longerrightarrow} \,
A
\end{equation}
for some $\mu\geq 1$, and we may assume $\mu\leq n-\mu$. Moreover, for $i=0,1,2$, we have $g_i = (-1)^i\det \psi_i$ where $\psi_i$ denotes the submatrix obtained by deleting row $i$ of $\psi$.
\end{prop}

\begin{proof}
For the initial claim, we repeat the argument in the proof of \Cref{dim V = 2 - g_0 g_1 regular sequence}. As $\dim V =3$ and $\{f_0',f_1', f_2'\}$ is a basis, the matrix $B$ in (\ref{dim 3 transition matrix equation}) has rank $3$. Moreover, since $\{u^n,u^{n-1}v,\ldots,v^n\}$ is $\K$-linearly independent, from (\ref{dim 3 - syzygy on subideal}) it follows that $g_i\neq 0$ for $i=0,1,2$. Hence it suffices to show that $\hgt (g_0,g_1,g_2)\neq 1$, namely that $g_0,g_1,g_2$ have no common factor. However, any nontrivial shared factor could be factored from $S$ in (\ref{dim 3 -- S defn}) to produce a syzygy of smaller degree, which contradicts the assumptions of \Cref{General Setting}. Thus $\hgt (g_0,g_1,g_2) =2$, and so this is an Artinian ideal of $A$.

Now that $(g_0,g_1,g_2)$ is seen to be a Cohen-Macaulay ideal of codimension two, by the Hilbert-Burch theorem \cite[20.15]{Eisenbud} it must have a resolution with the shape of (\ref{dim 3 - res of (g0,g1,g2)}), and we also have $g_i = (-1)^i\det \psi_i$ for $i=0,1,2$. The last item which must be verified is that the resolution in (\ref{dim 3 - res of (g0,g1,g2)}) is minimal, for which we need only show that $(g_0,g_1,g_2)$ is minimally generated. Since $g_0,g_1,g_2$ are homogeneous of the same degree, it suffices to show that $\{g_0,g_1,g_2\}$ is $\K$-linearly independent. However, this follows easily from (\ref{dim 3 - syzygy on subideal}) as $\{u^n, u^{n-1}v, \ldots ,v^n\}$ is $\K$-linearly independent and $\rk B=3$.
\end{proof}

With (\ref{dim 3 - res of (g0,g1,g2)}), we note that $\psi$ is the syzygy matrix of $(g_0,g_1,g_2)$, with columns $\psi= [C_1\,|\, C_2]$ consisting of entries of degree $\mu$ and $n-\mu$, respectively.

\begin{rem}\label{dim 3 f' decomp}
Since $[g_0\, \,g_1\,\,g_2]^T$ is a syzygy on $(f_0',f_1',f_2')$, we note that $[f_0'\, \,f_1'\,\,f_2']^T$ is a syzygy on the ideal $(g_0,g_1,g_2)$. As such, $[f_0'\, \,f_1'\,\,f_2']^T$ belongs to the image of the syzygy matrix $\psi$, in \Cref{dim V = 3 - (g_0 g_1 g_2) Artinian so HB G2P}. Thus from (\ref{dim 3 - res of (g0,g1,g2)}) and noting that $f_0',f_1',f_2' \in R_{a,b}$, we have
\begin{equation}\label{dim 3 - f alpha decomposition}
\left[\begin{array}{c}
     f_0'  \\[0.5ex]
     f_1'\\[0.5ex]
     f_2'
\end{array}\right] = \psi \left[\begin{array}{c}
    \alpha_1\\
    \alpha_2
\end{array}\right]
\end{equation}
for some bihomogeneous $\alpha_1 \in R_{a,b-\mu}$ and $\alpha_2\in R_{a,b-n+\mu}$.
\end{rem}

With the resolution of $(g_0,g_1,g_2)$ in (\ref{dim 3 - res of (g0,g1,g2)}), write $\psi = [C_1\,|\,C_2]$ once more to denote the  columns of this syzygy matrix. As noted in \Cref{dim 3 f' decomp}, the resolution in (\ref{dim 3 - res of (g0,g1,g2)}) is minimal, hence neither $C_1$ nor $C_2$ can have unit entries. Writing $I(C_1)$ and $I(C_2)$ to denote the ideal of entries of these columns, recall from \Cref{dim V = 3 - (g_0 g_1 g_2) Artinian so HB G2P} that $(g_0,g_1,g_2)$ is determinantal and is hence contained in both $I(C_1)$ and $I(C_2)$. Hence both $I(C_1)$ and $I(C_2)$ are proper ideals of $A=\K[u,v]$ with maximal codimension.

\begin{rem}\label{I(C1) and I(C2) HB resolutions}
 From the discussion above, both $I(C_1)$ and $I(C_2)$ admit Hilbert-Burch resolutions, and hence may be realized as determinantal ideals of almost-square matrices. Letting $\varphi_1$ and $\varphi_2$ denote the $3\times 2$ Hilbert-Burch syzygy matrices of $I(C_1)$ and $I(C_2)$, we have free resolutions 
 \begin{equation}\label{res I(C1) and I(C2)}
 0 \longrightarrow \begin{array}{c}
 A(-\mu-\nu_1)\\
 \oplus \\
A(-2\mu+\nu_1)
\end{array} \overset{\varphi_1}{\longrightarrow} \,
A^3(-\mu)\overset{C_1^T}{\longrightarrow} \,
A,\quad\quad 
0 \longrightarrow \begin{array}{c}
 A(-n+\mu-\nu_2)\\
 \oplus \\
A(-2n+2\mu+\nu_2)
\end{array} \overset{\varphi_2}{\longrightarrow} \,
A^3(-n+\mu)\overset{C_2^T}{\longrightarrow} \,
A, 
\end{equation}
for some $\nu_1\leq \mu$ and $\nu_2\leq n-\mu$. Moreover, the entries of $C_1$ and $C_2$ are precisely the signed minors of $\varphi_1$ and $\varphi_2$, with the same conventions of \Cref{dim V = 3 - (g_0 g_1 g_2) Artinian so HB G2P}. 
\end{rem}

With this observation, we begin the search for additional syzygies on $I_U$. First however, we note that $\alpha_1$ and $\alpha_2$ may be written in terms of the syzygy matrices in (\ref{res I(C1) and I(C2)}).

\begin{prop}\label{dim 3 alpha decomp}
With the assumptions of \Cref{General Setting}, assume that $\dim V=3$. There exist bihomogeneous polynomials $q_0\in R_{a,b-n-\nu_1}$ and $q_1 \in R_{a,b-\mu-n+\nu_1}$ such that $\alpha_1 = C_2^T\varphi_1 [q_0\,\,q_1]^T$. Moreover, there exist $r_0\in R_{a,b-n-\nu_2}$ and $r_1\in R_{a, b-2n+\mu+\nu_2}$ such that $\alpha_2 = C_1^T\varphi_2 [r_0\,\,r_1]^T$.
\end{prop}

\begin{proof}
We first note that these degrees are nonnegative, which follows as $n\geq 2$ by \Cref{low n dimension remark}, $b\geq 2n-1$ by assumption in \Cref{General Setting}, and since $\nu_1\leq \mu$, $\nu_2\leq n-\mu$, and $1\leq \mu \leq n-1$ by \Cref{I(C1) and I(C2) HB resolutions} and \Cref{dim V = 3 - (g_0 g_1 g_2) Artinian so HB G2P}. We prove the assertion for $\alpha_1$ as the claim for $\alpha_2$ follows similarly.

We first show that the entries of $C_2^T\varphi_1$ are coprime. Recall that the entries of these matrices are homogeneous polynomials in $A=\K[u,v]$. From (\ref{dim 3 - res of (g0,g1,g2)}) we have that $C_2$ consists of entries in degree $n-\mu$ and by (\ref{res I(C1) and I(C2)}), $\varphi_1$ consists of entries in degree $\nu_1$ and $\mu-\nu_1$ in the first and second column, respectively. With this, we may write $C_2^T\varphi_1 = [h_0\,\,h_1]$ where $h_0, h_1$ are homogeneous polynomials in $A$ with $\deg h_0 = n-\mu+\nu_1$ and $\deg h_1= n- \nu_1$. To show that $h_0$ and $h_1$ are coprime, write $G= \gcd (h_0,h_1)$ with $d=\deg G$. As such, there exist homogeneous $h_0',h_1'\in A$ such that $h_0 = h_0'G$ and $h_1=h_1'G$.

With this observation, notice that
$$ 0 = [h_0\,\,h_1]\begin{bmatrix}
    h_1'\\
    -h_0'
\end{bmatrix} = C_2^T\varphi_1 \begin{bmatrix}
    h_1'\\
    -h_0'
\end{bmatrix}$$
and so it follows that $\varphi_1 [h_1'\,\,-h_0']^T$ is a syzygy on $C_2^T$. However, recall that $\varphi_1$ is the syzygy matrix of $C_1^T$ following \Cref{I(C1) and I(C2) HB resolutions}, hence $\varphi_1 [h_1'\,\,-h_0']^T$ is a syzygy on $C_1^T$ as well. As $C_1$ and $C_2$ are the columns of $\psi$, it follows that $\varphi_1 [h_1'\,\,-h_0']^T \in \ker \psi^T$. Applying the functor $\Hom_A(-,A)$ and dualizing (\ref{dim 3 - res of (g0,g1,g2)}), one obtains the complex
\begin{equation}\label{dim 3 - dualized res of (g0,g1,g2)}
 0 \longrightarrow 
 A
 \overset{[g_0\,\,g_1\,\,g_2]^T}{\longerrightarrow}
 A^3(n)
 \overset{\psi^T}{\longrightarrow}
 \begin{array}{c}
 A(n+\mu)\\
 \oplus \\
A(2n-\mu)
\end{array}
\end{equation}
which remains exact by the Buchsbaum-Eisenbud acyclicity criterion \cite[Cor. 1]{BE73} and \Cref{dim V = 3 - (g_0 g_1 g_2) Artinian so HB G2P}. Hence $\ker \psi^T = \im [g_0\,\,g_1\,\,g_2]^T$, and so $\varphi_1 [h_1'\,\,-h_0']^T$ is a multiple of $[g_0\,\,g_1\,\,g_2]^T$. However, the entries of $\varphi_1 [h_1'\,\,-h_0']^T$ have degree $n-d$, where $d=\deg G$, and $g_0,g_1,g_2$ have degree $n$. Thus it follows that $d=0$, i.e. $G= \gcd(h_0,h_1)$ is a unit. Hence the entries of $C_2^T\varphi_1 = [h_0\,\,h_1]$ are indeed coprime.

With this, we verify the original assertion for $\alpha_1$. Writing $C_2^T\varphi_1 = [h_0\,\,h_1]$ once more, recall that $h_0, h_1$ are homogeneous polynomials in $A$ with $\deg h_0 = n-\mu+\nu_1$ and $\deg h_1= n- \nu_1$. Hence by \Cref{Sylvester containment lemma}, we have $(u,v)^{2n-\mu-1} \subseteq (h_0,h_1)$ in $A$, since $\gcd(h_0,h_1)=1$. Thus in $R$, since $\alpha_1 \in R_{a,b-\mu}$ and $b\geq 2n-1$, we have 
$$\alpha_1 \in (u,v)^{b-\mu} \subseteq (u,v)^{2n-\mu-1} \subseteq (h_0,h_1).$$
Hence we may write $\alpha_1 = q_0h_0+q_1h_1 = C_2^T\varphi_1 [q_0\,\,q_1]^T$ for some $q_0\in R_{a,b-n-\nu_1}$ and $q_1 \in R_{a,b-\mu-n+\nu_1}$. A similar argument verifies the claim for $\alpha_2$. 
\end{proof}

With this, we will show that the subideal $J=(f_0',f_1',f_2')\subseteq I_U$ is a perfect $R$-ideal of grade two, and hence has a Hilbert-Burch resolution. In particular, $\syz(J)$ then has only one additional generator, in addition to (\ref{dim 3 - syzygy on subideal}), which may be extended to a syzygy on $I_U$. First however, we state a couple of necessary lemmas, the first of which will be particularly useful for tensor product surfaces free of basepoints.

\begin{lemma}[{\cite[V.1.4.3]{Hartshorne}}]\label{Hartshorne lemma}
With $R=\K[s,t,u,v]$ as in \Cref{General Setting}, let $f\in R_{a,b}$ and $g\in R_{c,d}$ be bihomogeneous polynomials with $\gcd(f,g)=1$. The curves $V(f)$ and $V(g)$ in $\P^1\times \P^1$ intersect at $ad+bc$ points, counting multiplicity.
\end{lemma}

As $U$ is basepoint-free, \Cref{Hartshorne lemma} implies that the only proper two-generated ideals containing $I_U$ belong to either the subring $\K[u,v]$ or $\K[s,t]$. Hence \Cref{Hartshorne lemma} will prove useful for the duration of the article. The second lemma we require is less general, however it will be used multiple times throughout.

\begin{lemma}\label{Linear algebra lemma - minors of minors}
Let $A=(a_{ij})$ be a $3\times 2$ matrix with entries in a ring $R$. Write $\delta_i = (-1)^i\det A_i$ where $A_i$ denotes the submatrix of $A$ obtained by deleting row $i$, for $i=0,1,2$. For a $3\times 1$ column $D$ with entries in $R$, consider the $3\times 2$ matrix
\[
B=\left[\begin{array}{c|c}
    & \delta_0  \\
    D&\delta_1 \\
    &\delta_2
\end{array}\right]_.
\]
Similarly, writing $\Delta_i = (-1)^i\det B_i$, we have
\[
\Delta_i = \det \left[\begin{array}{cc}
    a_{i0} & a_{i1} \\[0.5ex]
    D^TC_1 & D^TC_2 
\end{array}\right]
\]
for $i=0,1,2$, where $C_1$ and $C_2$ denote the columns of $A=[C_1\,|\,C_2]$.
\end{lemma}

\begin{proof}
Notice that this determinant agrees with the product $D^T A[-a_{i1}\,\,a_{i0}]^T$. In particular, the product $A[-a_{i1}\,\,a_{i0}]^T$ has $0$ as its $i$th entry and the entries in position $j\neq i$ are precisely the minors involving the $i$th and $j$th rows of $A$. Hence it follows that $D^T A[-a_{i1}\,\,a_{i0}]^T$ agrees with the $i$th signed minor of $B$.
\end{proof}

With \Cref{Hartshorne lemma,Linear algebra lemma - minors of minors}, we may now describe the syzygies of the subideal $J=(f_0',f_1',f_2')$.

\begin{prop}\label{J g2p}
The $R$-ideal $J=(f_0',f_1',f_2')$ is perfect of grade two with syzygy matrix
\begin{equation}\label{HB subideal}
\Theta=\left[
\begin{array}{c|c}
     & g_0  \\
 \varphi_1\begin{bmatrix}
     q_0\\
     q_1
 \end{bmatrix} - \varphi_2\begin{bmatrix}
     r_0\\
     r_1
 \end{bmatrix}    & g_1\\
     &g_2
\end{array}
\right]
\end{equation}
for $\varphi_1$ and $\varphi_2$ in \Cref{I(C1) and I(C2) HB resolutions} and $q_0,q_1,r_0,r_1$  in \Cref{dim 3 alpha decomp}.    
\end{prop}

\begin{proof}
By the Hilbert-Burch theorem \cite[20.15]{Eisenbud}, it suffices to show that the generators of $J$ agree with the signed minors of $\Theta$, and also that $\hgt J = 2$. For the first assertion, we employ \Cref{Linear algebra lemma - minors of minors}. Consider the matrices
\[
P=\left[
\begin{array}{c|c}
     & g_0  \\
 \varphi_1\begin{bmatrix}
     q_0\\
     q_1
 \end{bmatrix}     & g_1\\
     &g_2
\end{array}
\right]
\quad\quad\text{and}\quad\quad
Q=\left[
\begin{array}{c|c}
     & g_0  \\
 \varphi_2\begin{bmatrix}
     r_0\\
     r_1
 \end{bmatrix}    & g_1\\
     &g_2
\end{array}
\right]
\]
and write $\Delta^P_i= (-1)^i\det P_i$ and $\Delta^Q_i = (-1)^i\det Q_i$ to denote their signed minors for $i=0,1,2$, following the previous convention. Recall from \Cref{dim V = 3 - (g_0 g_1 g_2) Artinian so HB G2P} that $g_0,g_1,g_2$ are the signed minors of $\psi = [C_1\,|\,C_2]$, hence by \Cref{Linear algebra lemma - minors of minors} it follows that
\[
\Delta^P_i = \det \left[
\begin{array}{cc}
    C_{1i}& C_{2i} \\[1ex]
[q_0 \,\, q_1]\varphi_1^T C_1 & [q_0\,\,q_1]\varphi_1^T C_2
\end{array}
\right]
\]
where $C_{1i}$ and $C_{2i}$ are the $i$th entries of $C_1$ and $C_2$. However, recall that $\varphi_1$ is the syzygy matrix on $C_1^T$ by \Cref{I(C1) and I(C2) HB resolutions}, hence $[q_0 \,\, q_1]\varphi_1^T C_1 =0$. Thus $\Delta^P_i = C_{1i} \cdot [q_0\,\,q_1]\varphi_1^T C_2 = C_{1i} \alpha_1$, following \Cref{dim 3 alpha decomp}. 

Similarly, the signed minors of $Q$ are 
\[
\Delta^Q_i = \det \left[
\begin{array}{cc}
    C_{1i}& C_{2i} \\[1ex]
[r_0 \,\, r_1]\varphi_2^T C_1 & [r_0\,\,r_1]\varphi_2^T C_2
\end{array}
\right]
\]
by \Cref{Linear algebra lemma - minors of minors}. As $\varphi_2$ is the syzygy matrix of $C_2^T$, we have $[r_0\,\,r_1]\varphi_2^T C_2=0$. Thus we have $\Delta^Q_i = -C_{2i} \cdot [r_0 \,\, r_1]\varphi_2^T C_1 = - C_{2i} \alpha_2$ by \Cref{dim 3 alpha decomp}. By multilinearity of determinants, the signed minors of $\Theta$ are precisely $\Delta^\Theta_i = C_{1i} \alpha_1 +C_{2i} \alpha_2 = f_i'$ for $i=0,1,2$, by \Cref{dim 3 f' decomp}.

As the generators of $J$ agree with the signed minors of $\Theta$, the claim will follow once it has been shown that $\hgt J \geq 2$. As $J$ is nonzero, it is enough to show that $\hgt J \neq 1$. Suppose otherwise that $\hgt J=1$, and so $f_0',f_1',f_2'$ have a non-unit common factor $h$. As $I_U = (f_0',f_1',f_2',p_3)\subseteq (h,p_3)$, it is clear that $h$ and $p_3$ have no common factor, as $ \hgt I_U =2$. Moreover, since $h$ is bihomogeneous, we may write $\bideg h= (c,d)$ for either $c\geq 1$ or $d\geq 1$. From this ideal containment, we have $V(h,p_3)\subseteq V(I_U)$, and so by \Cref{Hartshorne lemma} it follows that $V(I_U)$ contains $ad+bc >0$ points. However, this is a contradiction since $I_U$ is basepoint free. 
\end{proof}

\begin{cor}\label{dim 3 - alpha1 alpha2 reg sequence}
The polynomials $\alpha_1, \alpha_2$ form an $R$-regular sequence.   
\end{cor}

\begin{proof}
Notice that $J\subseteq (\alpha_1, \alpha_2)$, following \Cref{dim 3 f' decomp}. Hence by \Cref{J g2p}, we have $\hgt (\alpha_1,\alpha_2) = 2$.
\end{proof}

Following \Cref{J g2p}, the syzygy module of $J=(f_0',f_1',f_2')$ is spanned by the columns of $\Theta$ in (\ref{HB subideal}). Moreover, we note that these columns may be extended to syzygies $S$, as in (\ref{dim 3 -- S defn}), and $S_1$ on $I_U=(f_0',f_1',f_2',p_3)$. As such, in order to produce additional nontrivial syzygies on $I_U$, we must involve the last generator $p_3$.

\begin{rem}\label{dim 3 -- new syzygies remark}
Recall that $p_3 \in R_{a,b}$ and $b\geq 2n-1$ by assumption. Thus from the proof of \Cref{dim 3 alpha decomp} and applying \Cref{Sylvester containment lemma}, we may write
\begin{equation}\label{dim 3 - p_3 decomp}
p_3= C_2^T\varphi_1 \begin{bmatrix}
    m_0\\m_1
\end{bmatrix}= C_1^T\varphi_2 \begin{bmatrix}
    n_0\\n_1
\end{bmatrix}    
\end{equation}
for some bihomogeneous $m_0 \in R_{a,b-n+\mu-\nu_1}$, $m_1\in R_{a,b-n+\nu_1}$ and $n_0\in R_{a,b-\mu-\nu_2}$, $n_1\in R_{a,b-n+\nu_2}$. With this, we have the following syzygies on $I_U=(f_0',f_1',f_2',p_3)$:
\begin{equation}\label{dim 3 - Syzygies S_2 and S_3}
S_2=\left[
\begin{array}{c}
     \varphi_1\begin{bmatrix}
         m_0\\m_1
     \end{bmatrix}  \\[2ex]
      \hline
      -\alpha_2
\end{array}
\right]
\quad\quad\text{and}\quad\quad S_3=\left[
\begin{array}{c}
     \varphi_2\begin{bmatrix}
         n_0\\n_1
     \end{bmatrix}  \\[2ex]
      \hline
      -\alpha_1
\end{array}
\right]
\end{equation}
consisting of bihomogeneous entries in $R_{a,b-n+\mu}$ and $R_{a,b-\mu}$ respectively.
\end{rem}

\begin{proof}
A brief computation shows that $S_2$ and $S_3$ are syzygies on $I_U = (f_0',f_1',f_2',p_3)$. From \Cref{dim 3 f' decomp} we have
$$[f_0'\,\,f_1'\,\,f_2'\,\,p_3] \, S_2 = [\alpha_1\,\,\alpha_2]\psi^T \varphi_1\begin{bmatrix}
       m_0\\m_1
    \end{bmatrix} - \alpha_2 p_3.$$
However, the columns of $\psi$ are $C_1$ and $C_2$, and $\varphi_1$ is the syzygy matrix on the entries of $C_1$ by \Cref{I(C1) and I(C2) HB resolutions}. Hence $\varphi_1$ annihilates the first row of $\psi^T$ and so
$$[\alpha_1\,\,\alpha_2]\psi^T \varphi_1\begin{bmatrix}
       m_0\\m_1
    \end{bmatrix} - \alpha_2 p_3 = \alpha_2 C_2^T\varphi_1\begin{bmatrix}
        m_0\\m_1
   \end{bmatrix} - \alpha_2 p_3 =0$$
following (\ref{dim 3 - p_3 decomp}). Thus $S_2$ is a syzygy on $I_U$, and a similar computation shows that $[f_0'\,\,f_1'\,\,f_2'\,\,p_3] \, S_3=0$ as well.
\end{proof}

With \Cref{dim 3 -- new syzygies remark}, we claim that the syzygies obtained so far are sufficient to determine the bigraded strand $\Z_{2a-1,b-1}$ of the approximation complex, and hence the implicit equation of the surface $X_U$ following \Cref{Botbol implicit eqn}. Before we verify this claim, we state a brief lemma from linear algebra.

\begin{lemma}\label{linear algebra lemma - alternating matrix of minors}
Let $A$ be an $n\times 2$ matrix with entries in a ring. The product
$$B=A\begin{bmatrix}
    0&1\\
    -1&0
\end{bmatrix}A^T$$
is an $n\times n$ alternating matrix with entries $B_{ij} = \det \left[\begin{array}{c}
    \,R_i\,\\
    \hline
    R_j
\end{array}\right]$ where $R_i$, $R_j$ are the $i$th and $j$th rows of $A$. 
\end{lemma}

\begin{proof}
Multiplying $A$ with the $2\times 2$ alternating matrix above results in the $n\times 2$ matrix with the columns of $A$ reversed, and the first negated. Thus the product of the $i$th row of this matrix and the $j$th column of $A^T$ is precisely the $2\times 2$ minor of $A$ above.    
\end{proof}

\begin{thm}\label{dim V = 3 - main theorem}
With the assumptions of \Cref{General Setting}, assume that $\dim V=3$. The ideal $I_U$ has syzygies $S$ in bidegree $(0,n)$, $S_1$ in bidegree $(a,b-n)$, $S_2$ in bidegree $(a,b-n+\mu)$ and $S_3$ in bidegree $(a,b-\mu)$ such that $\dim \langle S,S_1,S_2,S_3\rangle_{2a-1,b-1} = 2ab$.    
\end{thm}

\begin{proof}
As noted, the columns of $\Theta$ in \Cref{J g2p} are syzygies on $J=(f_0',f_1',f_2')$ and may be extended to syzygies $S$ in (\ref{dim 3 -- S defn}) and $S_1$ on $I_U$ by adding zero entries. With this, consider the matrix
\[
M=\left[
\begin{array}{c|c|c|c}
g_0 & & & \\
g_1&  \varphi_1\begin{bmatrix}
     q_0\\
     q_1
 \end{bmatrix} - \varphi_2\begin{bmatrix}
     r_0\\
     r_1
 \end{bmatrix} & \varphi_1\begin{bmatrix}
         m_0\\m_1
     \end{bmatrix}  & \varphi_2\begin{bmatrix}
         n_0\\n_1
     \end{bmatrix}  \\
 g_2 & & & \\[1ex]    
      \hline
 0&  0&  -\alpha_2 & -\alpha_1
\end{array}
\right]
\]
with columns $\{S,S_1,S_2,S_3\}$ for $S_2$ and $S_3$ in \Cref{dim 3 -- new syzygies remark}. To verify the assertion, we must show that $M_{2a-1,b-1}$ consists of $2ab$ linearly independent columns. The number of columns will follow from a counting argument similar to the proof of \Cref{dim V = 2 - main theorem}, and it suffices to show that $M_{2a-1,b-1}$ is injective for the independence.

We note that $M$ itself is not injective, and we claim its kernel is spanned by
\[
N=\left[
\begin{array}{c}
     (q_0m_1-q_1m_0) + (r_0n_1-r_1n_0) \\
      p_3\\
      -\alpha_1\\
      \alpha_2
\end{array}
\right]_.
\]
We first verify that $N\in \ker M$, which follows from a computation. Clearly, $N$ multiplies the last row of $M$ to $0$. Hence by writing $M'$ to denote the submatrix obtained by deleting this row, it suffices to show that $M'N=0$. From \Cref{dim 3 alpha decomp} and (\ref{dim 3 - p_3 decomp}) we note that
\begin{equation}
\begin{aligned}\label{N in kernel computation}
&\,\,\varphi_1\begin{bmatrix}
     q_0\\
     q_1
 \end{bmatrix}p_3 - \varphi_2\begin{bmatrix}
     r_0\\
     r_1
 \end{bmatrix}p_3 -\varphi_1\begin{bmatrix}
         m_0\\m_1
     \end{bmatrix} \alpha_1 +\varphi_2\begin{bmatrix}
         n_0\\n_1
     \end{bmatrix}\alpha_2\\[1ex]
     =&\,\, \varphi_1\left(\begin{bmatrix}
     q_0\\
     q_1
 \end{bmatrix}[m_0\,\,m_1] - \begin{bmatrix}
         m_0\\m_1
     \end{bmatrix} [q_0\,\,q_1]\right)\varphi_1^TC_2 - \varphi_2\left(\begin{bmatrix}
     r_0\\
     r_1
 \end{bmatrix}[n_0\,\,n_1] - \begin{bmatrix}
         n_0\\n_1
     \end{bmatrix}[r_0\,\,r_1]\right)\varphi_2^TC_1 \\[1ex]
 =&\,\, -(q_1m_0 - q_0m_1) \cdot \varphi_1\begin{bmatrix}
     0&1\\
     -1&0
 \end{bmatrix}\varphi_1^TC_2 + (r_1n_0 - r_0n_1)\cdot \varphi_2\begin{bmatrix}
     0&1\\
     -1&0
 \end{bmatrix}\varphi_2^TC_1. \\[1ex]
\end{aligned}
\end{equation}
From \Cref{linear algebra lemma - alternating matrix of minors}, we note that 
\[
\varphi_1\begin{bmatrix}
     0&1\\
     -1&0
 \end{bmatrix}\varphi_1^T
 \quad\quad\text{and}\quad\quad
 \varphi_2\begin{bmatrix}
     0&1\\
     -1&0
 \end{bmatrix}\varphi_2^T
\]
are the $3\times 3$ alternating matrices consisting of the minors of $\varphi_1$ and $\varphi_2$, respectively. However, recall that these minors are precisely the (signed) entries of $C_1$ and $C_2$ by \Cref{I(C1) and I(C2) HB resolutions}. Hence by multiplying by $C_2^T$ and $C_1^T$, the entries of these products are precisely the signed minors of $\psi = [C_1\,|\,C_2]$. Thus by \Cref{dim V = 3 - (g_0 g_1 g_2) Artinian so HB G2P}, it follows that 
\begin{equation}\label{Continuation of N inkernel computation}
\varphi_1\begin{bmatrix}
     0&1\\
     -1&0
 \end{bmatrix}\varphi_1^TC_2 = -\begin{bmatrix}
     g_0\\
      g_1\\
      g_2
  \end{bmatrix}
 \quad\quad\text{and}\quad\quad
 \varphi_2\begin{bmatrix}
     0&1\\
     -1&0
 \end{bmatrix}\varphi_2^TC_1 = \begin{bmatrix}
    g_0\\
      g_1\\
      g_2
  \end{bmatrix}
\end{equation}
keeping track of the signs throughout. With this, combining (\ref{N in kernel computation}) and (\ref{Continuation of N inkernel computation}) shows that $M'N=0$, hence $N\in \ker M$ as claimed.

To verify that $\ker M$ is generated by $N$, notice that since $MN=0$, we have the following complex of bigraded $R$-modules and bihomogeneous maps
\begin{equation}\label{dim 3 - MN complex}
 0 \rightarrow R(-2a,-2b+n) \overset{N}{\longrightarrow} \begin{array}{c}
 R(0,-n)\\
 \oplus \\
R(-a,-b+n)\\
 \oplus \\
R(-a,-b+n-\mu)\\
\oplus\\
R(-a,-b+\mu)
\end{array} \overset{M}{\longrightarrow} \,
R^4
\end{equation}
with $M$ and $N$ as above. Moreover, we claim that this complex is exact. However, this follows easily by the Buchsbaum-Eisenbud acyclicity criterion \cite[Cor. 1]{BE73}. Indeed, notice that $\Theta$ from \Cref{J g2p} is a submatrix of $M$, and so it follows easily that $\rk M =3$. Moreover, the ideal of entries of $N$ has codimension at least two by \Cref{dim 3 - alpha1 alpha2 reg sequence}. Hence (\ref{dim 3 - MN complex}) is exact, and so it follows that $\ker M$ is spanned by $N$.

Since (\ref{dim 3 - MN complex}) is a bigraded free exact sequence, we may consider its strand in bidegree $(2a-1,b-1)$. From degree considerations in each component of the bigrading in (\ref{dim 3 - MN complex}), and hence the bidegrees of the entries in $N$, it follows that $N_{2a-1,b-1} =0$ and so $M_{2a-1,b-1}$ is injective, as claimed. Now that the columns of $M_{2a-1,b-1}$ have been shown to be linearly independent, we need only count them to verify the original assertion. As the bidegrees of the entries of each syzygy are known and recorded in (\ref{dim 3 - MN complex}), we proceed as in the proof of \Cref{dim V = 2 - main theorem}. 

As $S$ has entries in bidegree $(0,n)$, it contributes
$$h^0(\O_{\P^1\times \P^1}(2a-1,b-n-1)) = 2a(b-n)$$
columns to $M_{2a-1,b-1}$. From \Cref{J g2p} and \Cref{dim 3 alpha decomp}, the syzygy $S_1$ has entries in bidegree $(a,b-n)$, and hence yields
$$h^0(\O_{\P^1\times \P^1}(a-1,n-1)) = an$$
columns of $M_{2a-1,b-1}$. Lastly, by \Cref{dim 3 -- new syzygies remark} the syzygies $S_2$ and $S_3$ consist of entries in bidegree $(a,b-n+\mu)$ and $(a,b-\mu)$, respectively. Hence $S_2$ and $S_3$ give rise to
$$h^0(\O_{\P^1\times \P^1}(a-1,n-\mu-1)) = a(n-\mu)\quad\quad\text{and}\quad\quad h^0(\O_{\P^1\times \P^1}(a-1,\mu-1)) = a\mu$$
columns of $M_{2a-1,b-1}$. Summing these figures, we see that $\dim \langle S,S_1,S_2,S_3\rangle_{2a-1,b-1} = 2ab$.
\end{proof}

\begin{cor}
With the assumptions of \Cref{dim V = 3 - main theorem}, the first differential $d_1$ of the bigraded strand $\Z_{2a-1,b-1}$ of the approximation complex $\Z$ on the generators of $I_U$ is determined by the syzygies $\{S,S_1,S_2,S_3\}$. 
\end{cor}

\begin{proof}
This follows from \Cref{dim V = 3 - main theorem}, \Cref{Botbol det Z}, and \Cref{Botbol implicit eqn}, noting that $\dim \langle S,S_1,S_2,S_3\rangle_{2a-1,b-1}$ is precisely the number of columns contributed by these syzygies to a matrix representation of $d_1$.
\end{proof}

With this and \Cref{Botbol implicit eqn}, the determinant of the $2ab\times 2ab$ matrix representation of this differential is a power of the implicit equation of $X_U$.

\begin{rem}
We note that \Cref{dim V = 3 - main theorem} agrees with \cite[5.4]{Weaver25} when $n=2$. Note that if $n=2$, the transition matrix $B$ in (\ref{dim 3 transition matrix equation}) is the identity matrix, hence by (\ref{dim 3 - syzygy on subideal}) we have $[g_0\,\,g_1\,\,g_2]^T = [u^2\,\,uv\,\,v^2]^T$ with syzygy matrix
\[
\psi = \left[
\begin{array}{cc}
-v  & 0 \\
u  & -v\\
0  & u
\end{array}
\right]_.
\]
Writing $C_1$ and $C_2$ for the columns of $\psi$, the syzygy matrices of the entries of these columns are
\[
\begin{array}{cc}
   \varphi_1=\left[
   \begin{array}{ccc}
   0&u\\
   0&v\\
   1&0
   \end{array}
   \right]\quad\quad \text{and}\quad\quad
   &\varphi_2=\left[
   \begin{array}{cc}
   1&0\\
   0&u\\
   0&v\\
   \end{array}
   \right]_.
\end{array}
\]
With this, $C_2^T\varphi_1 =[u\,\,-v^2]$ and $C_1^T\varphi_2=[-v\,\,u^2]$, and $\alpha_1, \alpha_2$ may be written in terms of these matrices by \Cref{dim 3 alpha decomp}, which yields expressions similar to those in \cite[5.2]{Weaver25}. Similarly, $p_3$ may be written as in (\ref{dim 3 - p_3 decomp}) which is similar to its description in the proof of \cite[5.4]{Weaver25}. In particular, the syzygies of \Cref{dim V = 3 - main theorem} agree with those in \cite[5.4]{Weaver25} when $n=2$.
\end{rem}


\section{Syzygies in the case $\dim V=4$}\label{dim 4 section}

We now consider the final case of \Cref{Dimension of V} when the vector space $V$ associated to the singly graded syzygy $S$ has dimension $4$. Recall from \Cref{low n dimension remark} that $n\geq 3$ in this setting, hence this setting is untouched compared to the earlier work in \cite{DS16,Weaver25}. Many of the constructions here are analogous to and extend those developed in the previous sections. However, since $\dim V$ is maximal, we note that the constructions here are more involved than in the previous cases.

With the assumptions of \Cref{General Setting} throughout, we may construct the set $\{f_0,\ldots,f_n\}$ from $S$ as in (\ref{Syzygy rearrangement}). As we assume $\dim V=4$, there is a subset $\{f_0',f_1',f_2',f_3'\}$ among this set which forms a basis for $V$. We note that $V=U$ in this case, and so by \Cref{first dim V generators} we may take $I_U = (f_0',f_1',f_2',f_3')$. Moreover, by \Cref{B transition matrix and syzygies remark} there exists a $4\times (n+1)$ transition matrix $B$ with entries in $\K$ such that 
$$[f_0'\, \,f_1'\,\,f_2'\,\,f_3']\, B= [f_0\,\ldots\,f_n].$$
As before, from (\ref{Syzygy rearrangement}) we have that $[u^n\, u^{n-1}v\, \ldots \,v^n]^T$ is a syzygy on $(f_0,\ldots,f_n)$, hence
\begin{equation}\label{dim 4 - single graded syzygy}
S=B\, [u^n\, u^{n-1}v\, \ldots \,v^n]^T = \begin{bmatrix}
    g_0\\
    g_1\\
    g_2\\
    g_3
\end{bmatrix}    
\end{equation}
 is a syzygy on $I_U=(f_0',f_1',f_2',f_3')$, and is thus the singly graded syzygy in \Cref{General Setting}, with respect to this generating set, by \Cref{B transition matrix and syzygies remark}.

 \begin{prop}\label{dim V = 4 - (g_0 g_1 g_2 g_3) Artinian so HB G2P}
 With the assumptions of \Cref{General Setting}, assume that $\dim V =4$. With (\ref{dim 4 - single graded syzygy}), the ideal of entries $(g_0, g_1, g_2,g_3)$ is an Artinian ideal of $A=\K[u,v]$ with minimal free resolution
\begin{equation}\label{dim 4 - res of (g0,g1,g2,g_3)}
 0 \longrightarrow \begin{array}{c}
 A(-n-\mu_1)\\
 \oplus \\
 A(-n-\mu_2)\\
 \oplus\\
A(-2n+\mu_1+\mu_2)
\end{array} \overset{\psi}{\longrightarrow} \,
A^4(-n)\overset{[g_0\,\,g_1\,\,g_2\,\,g_3]}{\longestrightarrow} \,
A
\end{equation}
for some $\mu_1,\mu_2 \geq 1$, and we may assume $\mu_1 \leq \mu_2\leq n-\mu_1-\mu_2$. Moreover, for $0\leq i\leq 3$, we have $g_i = (-1)^i\det \psi_i$ where $\psi_i$ denotes the submatrix obtained by deleting row $i$ of $\psi$.
\end{prop}

\begin{proof}
This follows similarly to the proof of \Cref{dim V = 3 - (g_0 g_1 g_2) Artinian so HB G2P}.
\end{proof}

Here $\psi$ is the Hilbert-Burch syzygy matrix of $(g_0,g_1,g_2,g_3)$, with columns $\psi= [C_1\,|\, C_2\,|\,C_3]$ consisting of entries of degree $\mu_1$, $\mu_2$, and $n-\mu_1-\mu_2$. After possibly permuting these columns, we may assume $1\leq\mu_1 \leq \mu_2\leq n-\mu_1-\mu_2$.

\begin{rem}\label{dim 4 f' psi alpha}
As $S$ in (\ref{dim 4 - single graded syzygy}) is a syzygy on $I_U=(f_0',f_1',f_2',f_3')$, we note that $[f_0'\, \,f_1'\,\,f_2'\,\,f_3']^T$ is a syzygy on $(g_0,g_1,g_2,g_3)$. Hence by \Cref{dim V = 4 - (g_0 g_1 g_2 g_3) Artinian so HB G2P}, this syzygy belongs to the image of $\psi$ in (\ref{dim 4 - res of (g0,g1,g2,g_3)}). Thus we have 
\begin{equation}\label{dim 4 - f alpha decomposition}
\left[\begin{array}{c}
     f_0'  \\
     f_1'\\
     f_2'\\
     f_3'
\end{array}\right] = \psi \left[\begin{array}{c}
    \alpha_1\\
    \alpha_2\\
    \alpha_3
\end{array}\right]
\end{equation}
for some bihomogeneous $\alpha_1 \in R_{a,b-\mu_1}$, $\alpha_2\in R_{a,b-\mu_2}$, and $\alpha_3\in R_{a,b-n+\mu_1+\mu_2}$.
\end{rem}

With the resolution of $(g_0,g_1,g_2,g_3)$ in (\ref{dim 4 - res of (g0,g1,g2,g_3)}), write $\psi= [C_1\,|\, C_2\,|\,C_3]$ as above to denote its columns. As noted in \Cref{dim V = 4 - (g_0 g_1 g_2 g_3) Artinian so HB G2P}, the resolution (\ref{dim 4 - res of (g0,g1,g2,g_3)}) is minimal, hence $\psi$ has no unit entries. Thus the ideals of entries $I(C_1)$, $I(C_2)$, and $I(C_3)$ are proper ideals. Moreover, following \Cref{dim V = 4 - (g_0 g_1 g_2 g_3) Artinian so HB G2P}, each of these ideals contains $(g_0,g_1,g_2,g_3)$, and so these ideals of entries have maximal codimension in $A=\K[u,v]$.

\begin{rem}\label{dim 4 - I(C1) I(C2) I(C3) HB resolutions}
From the discussion above, the ideals $I(C_1)$, $I(C_2)$, and $I(C_3)$ are Artinian $A$-ideals and hence admit Hilbert-Burch resolutions. Write $\varphi_1$, $\varphi_2$, and $\varphi_3$ to denote the $4\times 3$ syzygy matrices of the ideals of entries of $C_1$, $C_2$, and $C_3$ respectively. With this, we have the following free resolutions:
\begin{equation}\label{dim 4 - res I(C1)}
     0 \longrightarrow \begin{array}{c}
 A(-\mu_1-\nu_{11})\\
 \oplus \\
A(-\mu_1-\nu_{12})\\
\oplus \\
A(-2\mu_1+\nu_{11}+\nu_{12})
\end{array} \overset{\varphi_1}{\longrightarrow} \,
A^4(-\mu_1)\overset{C_1^T}{\longrightarrow} \,
A,  
\end{equation}
\vspace{4mm}
 \begin{equation}\label{dim 4 - res I(C2)}
  0 \longrightarrow \begin{array}{c}
 A(-\mu_2-\nu_{21})\\
 \oplus \\
A(-\mu_2-\nu_{22})\\
\oplus \\
A(-2\mu_2+\nu_{21}+\nu_{22})
\end{array} \overset{\varphi_2}{\longrightarrow} \,
A^4(-\mu_2)\overset{C_2^T}{\longrightarrow} \,
A,
\end{equation}
\vspace{4mm}
\begin{equation}\label{dim 4 - res I(C3)}
 0 \longrightarrow \begin{array}{c}
 A(-n+\mu_1+\mu_2 - \nu_{31})\\
 \oplus \\
A(-n+\mu_1+\mu_2 - \nu_{32})\\
\oplus\\
A(-2n+2\mu_1+2\mu_2 + \nu_{31} + \nu_{32})
\end{array} \overset{\varphi_3}{\longrightarrow} \,
A^4(-n+\mu_1+\mu_2)\overset{C_3^T}{\longrightarrow} \,
A,
\end{equation}
for some integers $\nu_{11},\nu_{12}\leq \mu_1$, $\nu_{21},\nu_{22}\leq \mu_2$, and $\nu_{31},\nu_{32}\leq n-\mu_1-\mu_2$. Moreover, the entries of $C_1$, $C_2$, and $C_3$ are precisely the signed minors of their syzygy matrices above \cite[20.15]{Eisenbud}.
\end{rem}

Until now, we have followed the approach of the previous section. However, in the present setting we require yet another Hilbert-Burch structure involving the matrices above.

\begin{prop}\label{C_i^Tphi_j Artinian so HB}
For $i\neq j$, the entries of $C_i^T\varphi_j$ form an Artinian ideal of $A$. In particular, this ideal of entries admits a Hilbert-Burch resolution with $3\times 2$ syzygy matrix $\gamma_{ij}$, where the entries of $C_i^T\varphi_j$ are the signed minors of $\gamma_{ij}$.
\end{prop}

\begin{proof}
We first note that neither $C_i$ nor $\varphi_j$ has unit entries, hence the entries of their product form a proper ideal. As such, we have $\hgt I(C_i^T\varphi_j) \leq 2$ and so we need only show that $\hgt I(C_i^T\varphi_j) \neq 0,1$, and the claim will then follow from the Hilbert-Burch theorem \cite[20.15]{Eisenbud}. First note that if $C_i^T\varphi_j =0$, then the columns of $\varphi_j$ are syzygies on $C_i^T$. As $\varphi_i$ is the syzygy matrix of $C_i^T$, it follows that there is a $3\times 3$ matrix $M$ with entries in $A$ such that $\varphi_j = \varphi_i M$. With this, the minors of $\varphi_j$ agree with the minors of $\varphi_i$ multiplied by $\det M$. However, the signed minors of $\varphi_i$ and $\varphi_j$ are precisely the entries of $C_i$ and $C_j$ by \Cref{dim 4 - I(C1) I(C2) I(C3) HB resolutions}, hence it follows that the columns of $\psi$ are linearly dependent. However, this is a contradiction as $\rk \psi =3$ by \Cref{dim V = 4 - (g_0 g_1 g_2 g_3) Artinian so HB G2P}, and so $C_i^T\varphi_j \neq 0$.

We now show that $\hgt I(C_i^T\varphi_j) \neq 1$ for which, as before, it suffices to show that the entries of $C_i^T\varphi_j$ share no factor. Letting $G$ denote the greatest common factor among the entries of $C_i^T\varphi_j$, note that $G$ is a homogeneous polynomial in $A= \K[u,v]$ where $\K$ is algebraically closed. Hence if $\deg G \geq 1$, then $G$ splits as a product of homogeneous linear factors. Assume that $\deg G \geq 1$ and let $\ell$ denote one such linear factor, and write $\overline{\,\cdot\,}$ to denote images modulo the principal ideal $(\ell)$. Notice that $\overline{A}$ is a regular ring of dimension one, and is hence a principal ideal domain.

Modulo $(\ell)$, notice that $\overline{C_i}^T\overline{\varphi_j} =0$, and so $\overline{C_i} \in \ker \overline{\varphi_j}^T$. Moreover, since $\varphi_j$ is the syzygy matrix of $C_j^T$, we have $\overline{C_j} \in \ker \overline{\varphi_j}^T$ as well. Recall that $\rk \varphi_j =3$ and the signed minors of $\varphi_j$ are the entries of $C_j$, which form an ideal of height two by \Cref{dim 4 - I(C1) I(C2) I(C3) HB resolutions}. In particular, at least one of these minors is nonzero modulo $(\ell)$, and so $\rk \overline{\varphi_j} =3$. Thus, by additivity of rank, it follows that $\ker \overline{\varphi_j}^T$ has rank one. Moreover, as $\ker \overline{\varphi_j}^T \subseteq \overline{A}^4$ and $\overline{A}$ is a principal ideal domain, it follows that $\ker \overline{\varphi_j}^T$ is free of rank one. With this and since $\overline{C_i}, \overline{C_j} \in \ker \overline{\varphi_j}^T$, there is a dependency among these columns, and so $\rk \overline{\psi} \leq 2$. However, this is impossible as the minors of $\psi$ are the entries of $S$ in (\ref{dim 4 - single graded syzygy}), which form an ideal of height two by \Cref{dim V = 4 - (g_0 g_1 g_2 g_3) Artinian so HB G2P}. Thus at least one of these minors in nonzero modulo $(\ell)$, hence $\rk \overline{\psi} =3$. Thus no such linear factor $\ell$ of $G$ can exist, and so $G$ is a unit of $A$. Hence we have $\hgt I(C_i^T\varphi_j) \neq 1$.
\end{proof}

With the necessary technical framework in place, we begin our search for the additional syzygies of $I_U$ which, in addition to $S$, will determine the implicit equation. We begin by providing a description of the polynomials $\alpha_1$, $\alpha_2$, $\alpha_3$ in \Cref{dim 4 f' psi alpha} similar to that in \Cref{dim 3 alpha decomp}. First however, we state a brief lemma applicable to the short graded free resolutions above.

\begin{lemma}\label{HB degree shift lemma}
Let $R=\K[x_1,\ldots,x_d]$ be a standard graded polynomial ring for $d\geq 2$, and let $I$ be a homogeneous perfect $R$-ideal of codimension 2. With 
\[
0\longrightarrow \bigoplus_{i=1}^{n-1} R(-b_i) \overset{\varphi}{\longrightarrow} \bigoplus_{i=1}^{n} R(-a_i) \longrightarrow R 
\]
a free resolution of $R/I$, we have the following.
\begin{itemize}
    \item[(a)] The degree of each entry of a matrix representative of $\varphi = (\varphi_{ij})$ is $\deg \varphi_{ij} = b_j-a_i$.
\vspace{1mm}

    \item[(b)] We have $\sum_{i=1}^{n-1} b_i =  \sum_{i=1}^{n} a_i$.
\end{itemize}
\end{lemma}

\begin{proof}
Part (a) is clear from the degree shifts. For part (b), write $F= \bigoplus_{i=1}^{n} R(-a_i)$ and $G=\bigoplus_{i=1}^{n-1} R(-b_i)$. With the resolution of $R/I$ above, the Hilbert series of $R/I$ is 
$$h_{R/I}(t) = h_R(t) + h_G(t) - h_F(t) = \frac{1+\sum_{i=1}^{n-1} t^{b_i} - \sum_{i=1}^{n} t^{a_i} }{(1-t)^d}$$ 
following \cite[4.1.13]{BH93}. However, $I$ is a Cohen-Macaulay ideal of codimension two, hence $\dim R/I =d-2$. Thus it follows that $(1-t)^2$ divides $q(t) = 1+\sum_{i=1}^{n-1} t^{b_i} - \sum_{i=1}^{n} t^{a_i}$ \cite[4.1.8]{BH93}. From calculus, it then follows that $q(1)=0$ and also $q'(1)=0$. From this latter observation, it follows that $\sum_{i=1}^{n-1} {b_i} - \sum_{i=1}^{n} {a_i} =0$ and the claim follows.
\end{proof}

With this, the polynomials $\alpha_1$, $\alpha_2$, and  $\alpha_3$ in (\ref{dim 4 - f alpha decomposition}) may be written in terms of the syzygy matrices above.

\begin{prop}\label{dim 4 - alpha decomp}
    There exist bihomogeneous polynomials in $R$ such that
    \begin{equation}
        \begin{array}{c}
             \alpha_1=C_2^T \varphi_3\gamma_{13}\begin{bmatrix}
                 a_{20}\\
                 a_{21}
             \end{bmatrix}=C_3^T \varphi_2\gamma_{12}\begin{bmatrix}
                 a_{30}\\
                 a_{31}
             \end{bmatrix}\\[4ex]
             \alpha_2
             = C_3^T \varphi_2\gamma_{12}\begin{bmatrix}
                 b_{30}\\
                 b_{31}
             \end{bmatrix}=C_1^T \varphi_3\gamma_{23}\begin{bmatrix}
                 b_{10}\\
                 b_{11}
             \end{bmatrix}\\[4ex]
             \alpha_3 
             =C_2^T \varphi_3\gamma_{13}\begin{bmatrix}
                 c_{20}\\
                 c_{21}
             \end{bmatrix} 
             = C_1^T \varphi_3\gamma_{23}\begin{bmatrix}
                 c_{10}\\
                 c_{11}
             \end{bmatrix}
        \end{array}
    \end{equation}
for the matrices in \Cref{dim 4 - I(C1) I(C2) I(C3) HB resolutions} and \Cref{C_i^Tphi_j Artinian so HB}.    
\end{prop}

\begin{proof}
    We verify that $\alpha_1=C_2^T \varphi_3\gamma_{13}[a_{20}\,\,a_{21}]^T$ for some bihomogeneous $a_{20}, a_{21}$ as the other claims follow similarly. As noted in \Cref{C_i^Tphi_j Artinian so HB}, the entries of $C_1^T\varphi_3$ have a Hilbert-Burch resolution with syzygy matrix $\gamma_{13}$. As such, we may write
\begin{equation}\label{dim 4 - res C_1^Tvarphi_3}
 0 \longrightarrow \begin{array}{c}
 A(-P)\\
 \oplus \\
A(-Q)\\
\end{array} \overset{\gamma_{13}}{\longrightarrow} \,
\begin{array}{c}
 A(-\mu_1-\nu_{31})\\
 \oplus \\
A(-\mu_1 -\nu_{32})\\
\oplus\\
A(-n+\mu_2+\nu_{31}+\nu_{32})
\end{array}
\overset{C_1^T\varphi_3}{\longrightarrow} \,
A
\end{equation}
to denote this resolution over $A=\K[u,v]$, noting that the entries of $C_1^T$ have degree $\mu_1$ and the degrees of the entries of $\varphi_3$ along each column are $\nu_{31}$, $\nu_{32}$, and $n-\mu_1-\mu_2+\nu_{31}+\nu_{32}$ following (\ref{dim 4 - res I(C3)}).

Here $P$ and $Q$ denote the degree shifts in the last free module which depend on the degrees of the entries of $\gamma_{13}$, which we note are not constant along each column as the entries of $C_1^T\varphi_3$ have differing degrees. However, by \Cref{HB degree shift lemma} we note that $P+Q= n+2\mu_1-\mu_2$. Thus by computing the degrees of the entries of $\gamma_{13}$ with \Cref{HB degree shift lemma} and writing $C_2^T \varphi_3\gamma_{13} = [h_0\,\, h_1]$, it follows that $h_0, h_1$ are homogeneous polynomials in $A$ with $\deg h_0 = P-\mu_1+\mu_2$ and $\deg h_1 = Q-\mu_1+\mu_2$.

We claim that $\gcd(h_0,h_1) = 1$ and apply an argument similar to the proof of \Cref{dim 3 alpha decomp}. Write $G = \gcd(h_0,h_1)$ with $d= \deg G$, and so we have $h_0 = h_0'G$ and $h_1=h_1'G$ for some homogeneous $h_0',h_1'\in A$. With this observation, we note that 
$$ 0 = [h_0\,\,h_1]\begin{bmatrix}
    h_1'\\
    -h_0'
\end{bmatrix} = C_2^T \varphi_3\gamma_{13} \begin{bmatrix}
    h_1'\\
    -h_0'
\end{bmatrix}$$
from which it follows that $\varphi_3\gamma_{13}[h_1'\,\,-h_0']^T$ is a syzygy of $C_2^T$. However, notice that $\varphi_3\gamma_{13}[h_1'\,\,-h_0']^T$ is also a syzygy of $C_1^T$ and $C_3^T$ as well, since $\varphi_3$ is the syzygy matrix of $C_3^T$ and $\gamma_{13}$ is the syzygy matrix of $C_1^T\varphi_3$, following \Cref{dim 4 - I(C1) I(C2) I(C3) HB resolutions} and \Cref{C_i^Tphi_j Artinian so HB}. As such, it follows that $\varphi_3\gamma_{13}[h_1'\,\,-h_0']^T \in \ker \psi^T$, and an argument similar to the proof of \Cref{dim 3 alpha decomp} shows that $\ker \psi^T$ is spanned by $[g_0\,\,g_1\,\,g_2\,\,g_3]^T$. However, $\deg g_i = n$ and a brief calculation shows that $\varphi_3\gamma_{13}[h_1'\,\,-h_0']^T$ consists of entries in degree $n-d$, where $d= \deg G$. Thus it follows that $d=0$, and so $G=\gcd(h_0,h_1)$ is a unit.

Since $h_0$ and $h_1$ are coprime polynomials in $A=\K[u,v]$ with $\deg h_0 = P-\mu_1+\mu_2$ and $\deg h_1 = Q-\mu_1+\mu_2$, by \Cref{Sylvester containment lemma} and \Cref{HB degree shift lemma} with (\ref{dim 4 - res C_1^Tvarphi_3}), we have the containment
$$(u,v)^{n+\mu_2-1} = (u,v)^{P+Q-2\mu_1+2\mu_2-1} \subseteq (h_0,h_1).$$
Recall that $\alpha_1 \in R_{a,b-\mu_1}$, $b\geq 2n-1$, and $\mu_1\leq \mu_2 \leq n-\mu_1-\mu_2$ by \Cref{dim V = 4 - (g_0 g_1 g_2 g_3) Artinian so HB G2P}. Hence 
$$n+\mu_2-1 \leq n+(n-\mu_1-\mu_2) -1 \leq b-\mu_1-\mu_2\leq b-\mu_1$$ 
and so it follows that $\alpha_1 \in (h_0,h_1)$. Thus there exist bihomogeneous polynomials $a_{20}$, $a_{21}$ such that $\alpha_1 = a_{20}h_0+a_{21}h_1$, which may be written as the stated matrix product.
\end{proof}

From the proof of \Cref{dim 4 - alpha decomp}, we note that $a_{20}$ and $a_{21}$ have bidegrees $(a,b-P-\mu_2)$ and $(a,b-Q-\mu_2)$ respectively, where $P$ and $Q$ are as in (\ref{dim 4 - res C_1^Tvarphi_3}). Moreover, the proof above also verifies the first equality for $\alpha_3\in R_{a,b-n+\mu_1+\mu_2}$, noting that 
$$n+\mu_2-1 = (2n-1)-n+\mu_2 \leq  b- n+\mu_2 \leq b-n+\mu_1+\mu_2.$$
As such, we have $\alpha_3 \in (h_0,h_1)$ and the polynomials $c_{20}, c_{21}$ have bidegree $(a,b-P-n+2\mu_1)$ and $(a,b-Q-n+2\mu_1)$ respectively. We purposely omit these bidegrees in the statement of \Cref{dim 4 - alpha decomp} as they depend on these unknown degree shifts, and there is a similar phenomenon for the remaining identities. This will pose little issue, but we note that these bidegrees can be computed as above.

With the necessary constructions in place, we now employ them to produce additional syzygies on $I_U$.

\begin{prop}\label{dim 4 new syzygies}
With the notation in \Cref{dim 4 - alpha decomp}, we have the following syzygies on $I_U$:
\[
\begin{array}{ccc}
  S_1= \varphi_2\gamma_{12}\begin{bmatrix}
                 b_{30}\\
                 b_{31}
             \end{bmatrix}
             -\varphi_3\gamma_{13}\begin{bmatrix}
                 c_{20}\\
                 c_{21}
             \end{bmatrix}_,    \quad \quad
&  
S_2= \varphi_3\gamma_{23}\begin{bmatrix}
                 c_{10}\\
                 c_{11}
             \end{bmatrix}
             -\varphi_2\gamma_{12}\begin{bmatrix}
                 a_{30}\\
                 a_{31}
             \end{bmatrix}_,\quad\quad
&  
S_3= \varphi_3\gamma_{13}\begin{bmatrix}
                 a_{20}\\
                 a_{21}
             \end{bmatrix}
             -\varphi_3\gamma_{23}\begin{bmatrix}
                 b_{10}\\
                 b_{11}
             \end{bmatrix}_,
\end{array}
\]
consisting of entries in bidegree $(a,b-n+\mu_1)$, $(a,b-n+\mu_2)$, and $(a,b-\mu_1-\mu_2)$ respectively.
\end{prop}

\begin{proof}
The bidegrees of the entries follow from \Cref{dim 4 - alpha decomp} as the bidegrees of $\alpha_1$, $\alpha_2$, $\alpha_3$ and the degrees of the columns of $\psi$ are known. We verify that $S_1$ is a syzygy of $I_U= (f_0',f_1',f_2',f_3')$ and the remaining assertions will follow similarly. By (\ref{dim 4 - f alpha decomposition}) we note that
    \begin{equation}\label{S1 is a syzygy eqn 1}
        [f_0'\,\,f_1'\,\,f_2'\,\,f_3']\, S_1 =[\alpha_1\,\,\alpha_2\,\,\alpha_3]\psi^T \varphi_2\gamma_{12}\begin{bmatrix}
                 b_{30}\\
                 b_{31}
             \end{bmatrix}
             -[\alpha_1\,\,\alpha_2\,\,\alpha_3]\psi^T \varphi_3\gamma_{13}\begin{bmatrix}
                 c_{20}\\
                 c_{21}
             \end{bmatrix}_.  
             \end{equation}             
Recall that $\varphi_2$ is the syzygy matrix of $C_2^T$ and $\gamma_{12}$ is the syzygy matrix of $C_1^T\varphi_2$. Hence we see that $C_1^T\varphi_2\gamma_{12} = C_2^T\varphi_2\gamma_{12}=0$, and so $\varphi_2\gamma_{12}$ annihilates the first two rows of $\psi^T$. Similarly, it follows that $C_1^T\varphi_3\gamma_{13} = C_3^T\varphi_3\gamma_{13} =0$, and so $\psi^T\varphi_3\gamma_{13}$ has nonzero entries only in its second row. With this observation and (\ref{S1 is a syzygy eqn 1}), we have 
\[
[f_0'\,\,f_1'\,\,f_2'\,\,f_3']\, S_1=\alpha_3 C_3^T \varphi_2\gamma_{12}\begin{bmatrix}
                 b_{30}\\
                 b_{31}
             \end{bmatrix}
             -\alpha_2C_2^T \varphi_3\gamma_{13}\begin{bmatrix}
                 c_{20}\\
                 c_{21}
             \end{bmatrix} = \alpha_3\alpha_2 - \alpha_2\alpha_3 =0
\]
following \Cref{dim 4 - alpha decomp}. A similar computation shows that $S_2$ and $S_3$ are syzygies on $I_U$ as well.
\end{proof}

With the syzygies in \Cref{dim 4 new syzygies}, we claim that $\{S,S_1,S_2,S_3\}$ is sufficient to produce the first differential of the bigraded strand $\Z_{2a-1,b-1}$ of the approximation complex, and hence the implicit equation following \Cref{Botbol implicit eqn}. The main result of this section is as follows.

\begin{thm}\label{dim V = 4 - main theorem}
With the assumptions of \Cref{General Setting}, assume that $\dim V=4$. The ideal $I_U$ has syzygies $S$ in bidegree $(0,n)$, $S_1$ in bidegree $(a,b-n+\mu_1)$, $S_2$ in bidegree $(a,b-n+\mu_2)$ and $S_3$ in bidegree $(a,b-\mu_1-\mu_2)$ such that $\dim \langle S,S_1,S_2,S_3\rangle_{2a-1,b-1} = 2ab$. 
\end{thm}

\begin{proof}
Similar to the proof of \Cref{dim V = 3 - main theorem}, consider the matrix $M=[S\,|\,S_1\,|\,S_2\,|\,S_3]$ for $S$ as in (\ref{dim 4 - single graded syzygy}) and $S_1$, $S_2$, and $S_3$ as in \Cref{dim 4 new syzygies}. We must show that $M_{2a-1,b-1}$ consists of $2ab$ linearly independent columns. For the independence, it suffices to show that $M_{2a-1,b-1}$ is injective. We note that $M$ is not injective and, similar to the previous section, we describe the kernel of $M$ and show it vanishes in bidegree $(2a-1,b-1)$. We claim that the kernel of $M$ is spanned by 
\begin{equation}\label{dim 4 - N defn}
N=\left[
\begin{array}{c}
     H\\
      \alpha_1\\
      \alpha_2\\
      \alpha_3
\end{array}
\right]
\end{equation}
where $H= -(a_{20}c_{21}-a_{21}c_{20})+(b_{10}c_{11}-b_{11}c_{10})-(a_{30}b_{31}-a_{31}b_{30})$, which we verify first.

We note that $N$ consists of bihomogeneous entries of $R$. In particular, $H$ is a bihomogeneous polynomial of bidegree $(2a,2b-2n)$, which follows from a brief computation similar to the proof of \Cref{dim 4 - alpha decomp} and the discussion following it. As before, we show that $N \in \ker M$ and then apply the acyclicity criterion \cite[Cor. 1]{BE73} to show that it actually generates $\ker M$. To verify that $MN=0$, consider the following computation. From \Cref{dim 4 new syzygies}, we note that

\begin{equation}\label{alpha1S+alpha2S+alpha3S computation}
    \begin{aligned}
     \alpha_1S_1 + \alpha_2S_2+\alpha_3S_3 =&
\,\left(\varphi_2\gamma_{12}\begin{bmatrix}
                 b_{30}\\
                 b_{31}
             \end{bmatrix}
             -\varphi_3\gamma_{13}\begin{bmatrix}
                 c_{20}\\
                 c_{21}
             \end{bmatrix}\right) \alpha_1
+\left(\varphi_3\gamma_{23}\begin{bmatrix}
                 c_{10}\\
                 c_{11}
             \end{bmatrix}
             -\varphi_2\gamma_{12}\begin{bmatrix}
                 a_{30}\\
                 a_{31}
             \end{bmatrix}\right)\alpha_2\\[1ex] 
&\,+\left(\varphi_3\gamma_{13}\begin{bmatrix}
                 a_{20}\\
                 a_{21}
             \end{bmatrix}
             -\varphi_3\gamma_{23}\begin{bmatrix}
                 b_{10}\\
                 b_{11}
             \end{bmatrix}\right)\alpha_3\\[1ex]         
=&\,\left(
\varphi_3\gamma_{23}\begin{bmatrix}
                 c_{10}\\
                 c_{11}
             \end{bmatrix}\alpha_2 -
             \varphi_3\gamma_{23}\begin{bmatrix}
                 b_{10}\\
                 b_{11}
             \end{bmatrix} \alpha_3
\right)+
\left(
\varphi_3\gamma_{13}\begin{bmatrix}
                 a_{20}\\
                 a_{21}
             \end{bmatrix}\alpha_3
             -\varphi_3\gamma_{13}\begin{bmatrix}
                 c_{20}\\
                 c_{21}
             \end{bmatrix} \alpha_1 
\right)\\[1ex]
&\,+
\left(\varphi_2\gamma_{12}\begin{bmatrix}
                 b_{30}\\
                 b_{31}
             \end{bmatrix}\alpha_1-\varphi_2\gamma_{12}\begin{bmatrix}
                 a_{30}\\
                 a_{31}
             \end{bmatrix}\alpha_2\right)
    \end{aligned}
\end{equation}
and we claim that each of these summands is a multiple of $S$. We verify this for the first summand as a similar computation holds for the remaining ones. From \Cref{dim 4 - alpha decomp}, we have
\begin{equation}\label{alpha_iS_i sum computation}
\begin{aligned}
    \varphi_3\gamma_{23}\begin{bmatrix}
                 c_{10}\\
                 c_{11}
             \end{bmatrix}\alpha_2 -
             \varphi_3\gamma_{23}\begin{bmatrix}
                 b_{10}\\
                 b_{11}
             \end{bmatrix} \alpha_3&=\varphi_3\gamma_{23}\begin{bmatrix}
                 c_{10}\\
                 c_{11}
             \end{bmatrix} [b_{10}\,\,b_{11}]
             \gamma_{23}^T\varphi_3^TC_1 -
             \varphi_3\gamma_{23}\begin{bmatrix}
                 b_{10}\\
                 b_{11}
             \end{bmatrix} [c_{10}\,\,c_{11}]
             \gamma_{23}^T\varphi_3^TC_1 \\[1ex]
             &=(\varphi_3\gamma_{23})\left(\begin{bmatrix}
                 c_{10}\\
                 c_{11}
             \end{bmatrix} [b_{10}\,\,b_{11}] - \begin{bmatrix}
                 b_{10}\\
                 b_{11}
             \end{bmatrix} [c_{10}\,\,c_{11}]\right)(\varphi_3\gamma_{23})^TC_1\\[1ex]
             &=(b_{11}c_{10}-b_{10}c_{11})\cdot (\varphi_3\gamma_{23})\begin{bmatrix}
               0&1\\-1&0  
             \end{bmatrix} (\varphi_3\gamma_{23})^TC_1.
\end{aligned}
\end{equation}

By \Cref{linear algebra lemma - alternating matrix of minors}, the product $(\varphi_3\gamma_{23})\begin{bmatrix}
               0&1\\-1&0  
\end{bmatrix} (\varphi_3\gamma_{23})^T$ is the $4\times 4$ alternating matrix whose $ij$ entry is the $2\times 2$ minor of $\varphi_3\gamma_{23}$ consisting of rows $i$ and $j$. However, we claim that this agrees with the matrix of complementary minors of $[C_2\,|\,C_3]$, i.e. the alternating matrix with $ij$-entry obtained by \textit{deleting} rows $i$ and $j$ of $[C_2\,|\,C_3]$. Write $(\varphi_3\gamma_{23})_{ij}$ to denote the minor of $\varphi_3\gamma_{23}$ consisting of rows $i$ and $j$. By the Cauchy-Binet formula, and noting from \Cref{C_i^Tphi_j Artinian so HB} that the minors of $\gamma_{23}$ are the complementary entries of $C_2^T\varphi_3$ (up to sign),
it follows that
    \[
(\varphi_3\gamma_{23})_{ij} = \det \left[
\begin{array}{c}
     R_i  \\
     \hline
     R_j\\
     \hline
     C_2^T\varphi_3      
\end{array}
\right]
    \]
where $R_i$, $R_j$ denote rows $i$ and $j$ of $\varphi_3$. However, note that $C_2^T\varphi_3$ is a combination of the rows of $\varphi_3$ with coefficients the entries of $C_2$. Thus the claim follows from multilinearity of determinants, and noting that the signed minors of $\varphi_3$ are precisely the entries of $C_3$ by \Cref{dim 4 - I(C1) I(C2) I(C3) HB resolutions}. 

From this discussion and (\ref{alpha_iS_i sum computation}), it follows that 
\begin{equation}\label{alpha_iS_i sum computation part 2}
 \varphi_3\gamma_{23}\begin{bmatrix}
                 c_{10}\\
                 c_{11}
             \end{bmatrix}\alpha_2 -
             \varphi_3\gamma_{23}\begin{bmatrix}
                 b_{10}\\
                 b_{11}
             \end{bmatrix} \alpha_3 = (b_{11}c_{10}-b_{10}c_{11})\cdot \Delta_{23}C_1
\end{equation}
where $\Delta_{23}$ is the alternating matrix with $ij$ entry the signed minor of $[C_2\,|\,C_3]$ obtained by deleting rows $i$ and $j$. With this, it follows easily that $\Delta_{23}C_1$ consists of the signed $3\times 3$ minors of $\psi$, i.e. $\Delta_{23}C_1 =S$ by \Cref{dim V = 4 - (g_0 g_1 g_2 g_3) Artinian so HB G2P}. As such, from (\ref{alpha_iS_i sum computation part 2}), it follows that 
\[
\varphi_3\gamma_{23}\begin{bmatrix}
                 c_{10}\\
                 c_{11}
             \end{bmatrix}\alpha_2 -
             \varphi_3\gamma_{23}\begin{bmatrix}
                 b_{10}\\
                 b_{11}
             \end{bmatrix} \alpha_3 = (b_{11}c_{10}-b_{10}c_{11}) S.
\]
A similar argument, and keeping track of the signs throughout, shows that the remaining summands in (\ref{alpha1S+alpha2S+alpha3S computation}) are precisely
\[
\begin{array}{l}
\varphi_3\gamma_{13}\begin{bmatrix}
                 a_{20}\\
                 a_{21}
             \end{bmatrix}\alpha_3
             -\varphi_3\gamma_{13}\begin{bmatrix}
                 c_{20}\\
                 c_{21}
             \end{bmatrix} \alpha_1  = (a_{20}c_{21}-a_{21}c_{20})S, \\[3ex]
\varphi_2\gamma_{12}\begin{bmatrix}
                 b_{30}\\
                 b_{31}
             \end{bmatrix}\alpha_1-\varphi_2\gamma_{12}\begin{bmatrix}
                 a_{30}\\
                 a_{31}
             \end{bmatrix}\alpha_2= (a_{30}b_{31}-a_{31}b_{30})S.
\end{array}
\]
With this observation and (\ref{alpha1S+alpha2S+alpha3S computation}), it follows that $\alpha_1S_1 + \alpha_2S_2+\alpha_3S_3 = -HS$, and so we see that $N\in \ker M$.

Now that it has been shown that $MN=0$, we may form the following complex of bigraded $R$-modules and bihomogeneous maps
\begin{equation}\label{dim 4 - MN complex}
 0 \rightarrow R(-2a,-2b+n) \overset{N}{\longrightarrow} \begin{array}{c}
 R(0,-n)\\
 \oplus \\
R(-a,-b+n-\mu_1)\\
 \oplus \\
R(-a,-b+n-\mu_2)\\
\oplus\\
R(-a,-b+\mu_1+\mu_2)
\end{array} \overset{M}{\longrightarrow} \,
R^4
\end{equation}
which we claim is exact. We apply the acyclicity criterion \cite[Cor. 1]{BE73} once more, and need only verify that its conditions are met. Note that $I_U = (f_0',f_1',f_2',f_3') \subseteq (\alpha_1,\alpha_2,\alpha_3)$ by \Cref{dim 4 f' psi alpha}, and so the entries of $N$ form an ideal of height at least $2$ in $R$. Thus, since $R$ is a domain, we need only show that $\rk M=3$ to conclude that (\ref{dim 4 - MN complex}) is exact and so $\ker M$ is spanned by $N$. It suffices to show that $\rk M^T =3$, which will be the simpler computation. 

Notice that, since the columns of $M$ are syzygies on $I_U$, we have $[f_0'\,\,f_1'\,\,f_2'\,\,f_3']^T \in \ker M^T$. We claim that this actually generates $\ker M^T$. Indeed, for any $\omega\in \ker M^T$ notice that, since $S^T$ is the first row of $M^T$, $\omega$ must belong to the image of its syzygy matrix $\psi$. Hence we may write
\begin{equation}\label{omega matrix description}
\omega = \psi \left[\begin{array}{c}
    \beta_1\\
    \beta_2\\
    \beta_3
\end{array}\right]
\end{equation}
for some bihomogeneous $\beta_1, \beta_2, \beta_3$, similar to the expression in (\ref{dim 4 - f alpha decomposition}). With this, an argument similar to the proof of \Cref{dim 4 new syzygies} shows that 
\begin{equation}\label{MT omega zero}
0 = M^T\omega = \left[\begin{array}{c}
0 \\[0.25ex]
\alpha_2\beta_3- \alpha_3\beta_2\\[0.25ex]
\alpha_3\beta_1-\alpha_1\beta_3\\[0.25ex]
\alpha_1\beta_2- \alpha_2\beta_1
\end{array}
\right]_.
\end{equation}

We note that $\alpha_1,\alpha_2$ forms a regular sequence. As before, writing $h=\gcd(\alpha_1,\alpha_2)$, it suffices to show that $h$ is a unit. From (\ref{dim 4 - f alpha decomposition}) we have $I_U\subseteq (\alpha_1,\alpha_2,\alpha_3) \subseteq (h,\alpha_3)$, hence $V(h,\alpha_3) \subseteq V(I_U) = \emptyset$, since $I_U$ is basepoint-free. Thus by \Cref{Hartshorne lemma}, it follows that $h$ must have bidegree $(0,0)$ and is hence a unit. Since $\alpha_1, \alpha_2$ is a regular sequence and $\alpha_1\beta_2- \alpha_2\beta_1 =0$ in (\ref{MT omega zero}), it follows that $\beta_1 = \lambda\alpha_1$ and $\beta_2 = \lambda\alpha_2$ for some $\lambda\in R$. Moreover, since $\alpha_3\beta_1 - \alpha_1\beta_3=0$ and $R$ is a domain, it follows that $\beta_3= \lambda\alpha_3$ as well. Thus from (\ref{omega matrix description}) and \Cref{dim 4 f' psi alpha}, it follows that $\omega = \lambda[f_0'\,\,f_1'\,\,f_2'\,\,f_3']^T$ and so $\ker M^T$ is indeed spanned by $[f_0'\,\,f_1'\,\,f_2'\,\,f_3']^T$. Thus we have the bigraded free complex
\begin{equation}\label{M^T complex}
 0 \rightarrow R(-a,-b) \overset{[f_0'\,\,f_1'\,\,f_2'\,\,f_3']^T}{\longestrightarrow}
 R^4 \overset{M^T}{\longrightarrow} \,
\begin{array}{c}
 R(0,n)\\
 \oplus \\
R(a,b-n+\mu_1)\\
 \oplus \\
R(a,b-n+\mu_2)\\
\oplus\\
R(a,b-\mu_1-\mu_2)
\end{array}
\end{equation}
which is exact. Thus we have $\rk M^T=3$ by \cite[Cor. 1]{BE73}, as claimed. Hence $\rk M=3$ and so (\ref{dim 4 - MN complex}) is exact, by \cite[Cor. 1]{BE73} once more.

As (\ref{dim 4 - MN complex}) is a bigraded free exact sequence, we may consider its strand in bidegree $(2a-1,b-1)$. From degree considerations in each component of the bigrading, and the bidegrees of the entries of $N$, it follows that $N_{2a-1,b-1} =0$. Thus $M_{2a-1,b-1}$ is injective, and so its columns are linearly independent. Thus we need only count them to verify the original assertion. As the bidegrees of each syzygy are recorded in (\ref{dim 4 - MN complex}), we proceed as in the proofs of \Cref{dim V = 2 - main theorem} and \Cref{dim V = 3 - main theorem}.

As before, $S$ has entries in bidegree $(0,n)$, hence it yields
$$h^0(\O_{\P^1\times \P^1}(2a-1,b-n-1)) = 2a(b-n)$$
columns of $M_{2a-1,b-1}$. By \Cref{dim 4 new syzygies}, syzygies $S_1$, $S_2$, and $S_3$ consist of entries in bidegree $(a,b-n+\mu_1)$, $(a,b-n+\mu_2)$, and $(a,b-\mu_1-\mu_2)$ respectively. Thus $S_1$ contributes
$$h^0(\O_{\P^1\times \P^1}(a-1,n-\mu_1-1))=a(n-\mu_1)$$
columns to $M_{2a-1,b-1}$, the syzygy $S_2$ yields
$$h^0(\O_{\P^1\times \P^1}(a-1,n-\mu_2-1))= a(n-\mu_2)$$
columns of $M_{2a-1,b-1}$, and lastly the syzygy $S_3$ gives rise to
$$h^0(\O_{\P^1\times \P^1}(a-1,\mu_1+\mu_2-1)) = a(\mu_1+\mu_2)$$
columns of $M_{2a-1,b-1}$. Summing these figures, we have $\dim \langle S,S_1,S_2,S_3\rangle_{2a-1,b-1} = 2ab$.
\end{proof}

\begin{cor}
With the assumptions of \Cref{dim V = 4 - main theorem}, the first differential $d_1$ of the bigraded strand $\Z_{2a-1,b-1}$ of the approximation complex $\Z$ on the generators of $I_U$ is determined by the syzygies $\{S,S_1,S_2,S_3\}$. 
\end{cor}

\begin{proof}
This follows from \Cref{dim V = 4 - main theorem}, \Cref{Botbol det Z}, and \Cref{Botbol implicit eqn}, noting that $\dim \langle S,S_1,S_2,S_3\rangle_{2a-1,b-1}$ is precisely the number of columns contributed by these syzygies to a matrix representation of $d_1$.
\end{proof}

With this and \Cref{Botbol implicit eqn}, the determinant of the $2ab\times 2ab$ matrix representation of $d_1$ is a power of the implicit equation of $X_U$.

\begin{ex}\label{Example - dim V = 4}
Consider the basepoint-free tensor product surface parameterization $\phi_U:\,\P^1\times\P^1\rightarrow \P^3$ given by the subspace
$$U = \Span\{-s^2u^4v-s^2v^5,\,\,
s^2u^5-s^2u^3v^2+s^2uv^4-t^2v^5,\,\,
-t^2u^4v+t^2uv^4,\,\,
s^2u^5+t^2u^{5} \}$$
of $H^0(\O_{\P^1\times\P^1}(2,5))$. A brief computation shows that $S=[u^3\,\,u^2v\,\,uv^2\,\,v^3]^T$ is a syzygy on $I_U$, hence the matrix $A$ in \Cref{first dim V generators} is the identity. Thus we have $V=U$, and so clearly $\dim V=4$ in this setting. Moreover, as $n=3$ here, it follows that the transition matrix $B$ in \Cref{first dim V generators} is the identity as well. Hence we may construct the three remaining syzygies using the polynomials above as the basis for $V=U$.

Notice that the entries of $S$ are precisely the signed minors of 
\[
\psi = \left[\begin{array}{ccc}
-v &0&0 \\
u& -v&0\\
0&u&-v\\
0&0&u
\end{array}
\right]
\]
and moreover, $[f_0\,\,f_1\,\,f_2\,\,f_3] = \psi [\alpha_1\,\,\alpha_2\,\,\alpha_3]^T$ where $f_0,f_1,f_2,f_3$ are the polynomials spanning $U$ and
\[
\begin{array}{ccc}
     \alpha_1= t^2u^4+s^2v^4, \quad\quad &\alpha_2=2t^2v^4, \quad\quad &\alpha_3= s^2u^4+t^2v^4. 
\end{array}
\]
Writing $C_1, C_2, C_3$ to denote the columns of $\psi$, by \Cref{dim 4 - I(C1) I(C2) I(C3) HB resolutions} the entries of the columns may be realized as the signed minors of the following Hilbert-Burch syzygy matrices
\[
\begin{array}{ccc}
   \varphi_1=\left[
   \begin{array}{ccc}
   0&0&u\\
   0&0&v\\
   0&1&0\\
   1&0&0
   \end{array}
   \right]_,\quad\quad
   &\varphi_2=\left[
   \begin{array}{ccc}
   0&1&0\\
   0&0&u\\
   0&0&v\\
   1&0&0
   \end{array}
   \right]_,\quad\quad
   &\varphi_3=\left[
   \begin{array}{ccc}
   0&1&0\\
   1&0&0\\
   0&0&u\\
   0&0&v
   \end{array}
   \right]_.
\end{array}
\]
Moreover, we note that 
\[
\begin{array}{ccc}
C_1^T\varphi_2= [0\,\,\,-v\,\,\,u^2], \quad\quad & C_1^T\varphi_3= [u\,\,\,-v\,\,\,0], \quad\quad &C_2^T\varphi_3=[-v\,\,\,0\,\,\,u^2],
\end{array}
\]
which, following \Cref{C_i^Tphi_j Artinian so HB}, have Hilbert-Burch syzygy matrices 
\[
\begin{array}{ccc}
   \gamma_{12}=\left[
   \begin{array}{ccc}
   1&0\\
   0&u^2\\
   0&v
   \end{array}
   \right]_,\quad\quad
   &\gamma_{13}=\left[
   \begin{array}{ccc}
   0&v\\
   0&u\\
   -1&0
   \end{array}
   \right]_,\quad\quad
   &\gamma_{23}=\left[
   \begin{array}{ccc}
   0&u^2\\
   -1&0\\
   0&v
   \end{array}
   \right]_.
\end{array}
\]
With this and $\alpha_1, \alpha_2, \alpha_3$ above, we note that 
\[
\begin{array}{ll}
\alpha_1=C_2^T \varphi_3\gamma_{13}\begin{bmatrix}
a_{20}\\
a_{21}
\end{bmatrix}=[ -u^2\,\,\,-v^2 ]\begin{bmatrix}
-t^2u^2\\
-s^2v^2
\end{bmatrix}\quad\quad\quad
&
\alpha_1= C_3^T \varphi_2\gamma_{12}\begin{bmatrix}
 a_{30}\\
a_{31}
\end{bmatrix}=[u \,\,\,-v^3 ]\begin{bmatrix}
t^2u^3\\
-s^2v
\end{bmatrix}\\[4ex]
\alpha_2= C_3^T \varphi_2\gamma_{12}\begin{bmatrix}
b_{30}\\
b_{31}
\end{bmatrix}=[u \,\,\,-v^3 ]\begin{bmatrix}
0\\
-2t^2v
\end{bmatrix}
& \alpha_2 =C_1^T \varphi_3\gamma_{23}\begin{bmatrix}
 b_{10}\\
b_{11}
\end{bmatrix}=[ v\,\,\,u^3 ]\begin{bmatrix}
2t^2v^3\\
0
\end{bmatrix}\\[4ex]
\alpha_3 =C_2^T \varphi_3\gamma_{13}\begin{bmatrix}
c_{20}\\
c_{21}
\end{bmatrix} = [ -u^2\,\,\,-v^2 ]\begin{bmatrix}
-s^2u^2\\
-t^2v^2
\end{bmatrix}
&\alpha_3=C_1^T \varphi_3\gamma_{23}\begin{bmatrix}
c_{10}\\
c_{11}
\end{bmatrix}=[ v\,\,\,u^3 ]\begin{bmatrix}
t^2v^3\\
s^2u
\end{bmatrix}
\end{array}
\]
and so, following \Cref{dim 4 new syzygies}, we have
\[\begin{array}{ccc}
S_1 = \begin{bmatrix}
t^2uv^2-2t^2u^2v\\
t^2v^3-2t^2uv^2\\
-s^2u^3-2t^2v^3\\
-s^2u^2v
\end{bmatrix}_,\quad\quad
& S_2 = \begin{bmatrix}
 s^2u^2v - t^2v^3\\
 s^2u^3+s^2uv^2\\
 s^2u^2v+s^2v^3\\
 s^2uv^2-t^2u^3
\end{bmatrix}_,
\quad\quad
&S_3=\begin{bmatrix}
 2t^2v^3-s^2uv^2\\
 -s^2v^3\\
 t^2u^3\\
 t^2u^2v
\end{bmatrix}_.
\end{array}
\]
for the additional syzygies on $I_U$.

Write $\K[x_0,x_1,x_2,x_3]$ to denote the coordinate ring of $\P^3$ and let $M=[S\,|\,S_1\,|\,S_2\,|\,S_3]$ be as in the proof of \Cref{dim V = 4 - main theorem}.
Following \Cref{syzygies to approx comp} and \Cref{Botbol det Z}, we multiply $[x_0\,\,x_1\,\,x_2\,\,x_3]M$ and take this product in bidegree $(2a-1,b-1)$ of $R= \K[s,t,u,v]$. Hence by \Cref{dim V = 4 - main theorem} and its corollary, we obtain a $20\times 20$ matrix representation of the differential $d_1$ of $\Z_{2a-1,b-1}$, with entries in $\K[x_0,x_1,x_2,x_3]$. Moreover, a computation shows that $\det d_1 = F^2$ where $F$ is the implicit equation. In particular, $\deg F =10$ and by \Cref{Botbol implicit eqn} the degree of the rational map defined by $U$ is $\deg \phi_U =2$.
\end{ex}

Although the process outlined in the current and previous sections requires the computation of multiple syzygies, we note that the syzygies needed to construct $S_1,S_2,S_3$ have entries in $\K[u,v]$ and are all of Hilbert-Burch type. In particular, they are relatively quick to compute, especially compared to computing $\syz(I_U)$ over $R=\K[s,t,u,v]$. For instance, computations in \texttt{Macaulay2} \cite{Macaulay2} show that finding the implicit equation of $X_U$ in \Cref{Example - dim V = 4} using the constructions of \Cref{dim V = 4 - main theorem} is slightly faster than computing the entire syzygy module $\syz(I_U)$ and using it to produce $\Z_{2a-1,b-1}$. However, for larger $a$, $b$, and $n$, $\syz(I_U)$ typically becomes larger, whereas the number of syzygies required here is fixed, and these syzygies involve fewer variables. Further computations show that the methods presented here tend to be faster overall.

\section{Open questions}\label{Open questions section}

We conclude the article with some open questions stemming from the results presented here.

\begin{quest}
The assumption that $U$ is basepoint-free is actually used quite sparingly throughout. Regardless of basepoints, the constructions presented here are syzygies on $I_U$ and hence appear within the approximation complex. This raises the question of whether these constructions can be applied to determine the implicit equation of $X_U$ when basepoints are present. In particular, can one determine the implicit equation if $I_U$ has a singly graded syzygy and if there are basepoints present which are \textit{locally complete intersections}, or if the basepoints of $\phi_U$ are in generic position, as in \cite{Duarte17}?
\end{quest}

\begin{quest}
Within the assumptions of \Cref{General Setting}, we assume that the singly graded syzygy has minimal degree, in case there are multiple such syzygies. In \cite{DS16}, it is shown that basepoint-freeness ensures there can be at most one such syzygy when $n=1$. As the setting of \cite{DS16} is similar to the case when $\dim V=2$ in \Cref{dim 2 section}, it seems plausible that one might be able to extend this result. Does this assumption limit the number of singly graded minimal syzygies in general?
\end{quest}

\begin{quest}
In \cite{DS16}, Duarte and Schenck show that the singular locus of $X_U$ contains a line if $I_U$ has a linear syzygy. As the setting of \cite{DS16} is similar to the general case when $\dim V=2$ in \Cref{dim 2 section}, it seems natural to ask whether this phenomenon persists in general. Do the matrix representations of the implicit equation obtained here give insight into the singular locus of $X_U$?
\end{quest}

\begin{quest}
As noted in the introduction, the importance of a singly graded syzygy of a tensor product surface ideal $I_U$ was originally noticed by Schenck, Seceleanu, and Validashti in \cite{SSV14}. There, the presence of a linear syzygy in bidegree $(0,1)$ or $(1,0)$ was shown to impose strong conditions on a bigraded free resolution of $I_U$ when $U\subseteq H^0(\O_{\P^1\times \P^1}(2,1))$, motivating the further study by Duarte and Schenck in \cite{DS16}. Does a similar phenomenon occur in a free resolution of $I_U$ if it has a syzygy in bidegree $(0,n)$ or $(n,0)$ in general? In particular, the syzygies constructed here and the exact sequences in (\ref{dim 3 - MN complex}) and (\ref{dim 4 - MN complex}) may be taken as summands of such a resolution.
\end{quest}


\section*{Acknowledgements}

Computations using \texttt{Macaulay2} \cite{Macaulay2} were essential in the preparation of this article, providing numerous examples and verifying direct computations, leading to the constructions presented here.


\end{document}